\documentclass[11pt]{article}
\usepackage{amsmath,amssymb,amsthm,amsfonts,amstext,amsbsy,amscd}

\usepackage{mathrsfs}
\usepackage{graphics}
\usepackage{graphicx}
\usepackage{float}
\usepackage{multirow}
\usepackage{multicol}
\usepackage{latexsym}
\usepackage[firstinits=false,backend=bibtex,isbn=false,url=false,doi=false,maxbibnames=99, firstinits]{biblatex}

\bibliography{BiblioDKNote}
\providecommand{\Fourier}{\mathcal{F}}
\usepackage{diagbox}

\usepackage{dsfont}
\usepackage{color}
\usepackage[top=3.5cm, bottom=3.5cm, left=3cm, right=3cm]{geometry}

\RequirePackage[colorlinks,citecolor=blue,urlcolor=blue,linkcolor=green]{hyperref}
\renewcommand{\phi}{\varphi}
\DeclareMathOperator{\Log}{Log}

\newcommand{\PP}{\mathbb{P}}

\newcommand{\E}{\mathbb{E}}

\newcommand{\R}{\mathbb{R}}
\newcommand{\N}{\mathbb{N}}
\newcommand{\V}{\mathbb{V}}

\numberwithin{equation}{section}
\newtheorem{theorem}{Theorem}[section]
\newtheorem{lemma}{Lemma}[section]

\renewcommand{\d}{\, \textnormal{d}  }
\renewcommand{\tilde}{\widetilde}
\DeclareMathOperator{\lk}{L}
\renewcommand{\hat}{\widehat}

\providecommand{\ephi}{\hat{\phi} }
\providecommand{\dephi}{\hat{\phi}' }
\providecommand{\dphi}{{\phi}' }
\providecommand{\intm}{\int_{-m}^{m} }

\providecommand{\eps}{\varepsilon}

\begin{document}
\title{An adaptive procedure for Fourier estimators: illustration to deconvolution and decompounding}
\author{C\'eline Duval\thanks{MAP5, UMR CNRS
8145, Universit\'e Paris Descartes.}
 \hspace{0.1cm}and
Johanna Kappus\thanks{Institut f\"{u}r Mathematik, Universit\"{a}t Rostock.}
}
\date{}

\maketitle

\begin{abstract} 
We introduce a new procedure to select the optimal cutoff parameter for Fourier density estimators that leads to adaptive rate optimal estimators, up to a logarithmic factor. This adaptive procedure applies for different inverse problems. We illustrate it on two classical examples: deconvolution and decompounding, i.e. non-parametric estimation of the jump density of a compound Poisson process from the observation of $n$ increments of length $\Delta>0$. For this latter example, we first build an estimator for which we provide an upper bound for its $\lk^2$-risk that is valid simultaneously for sampling rates $\Delta$ that can vanish, $\Delta:=\Delta_{n}\to0$, can be fixed, $\Delta_{n}\to\Delta_{0}>0$ or can get large, $\Delta_{n}\to\infty$ slowly. This last result is new and presents interest on its own. Then, we show that the adaptive procedure we present leads to an adaptive and rate optimal (up to a logarithmic factor) estimator of the jump density. \end{abstract}

\section{Introduction}

\subsection{Motivation}

In the literature on non-parametric statistics, and in particular in the literature dedicated to building minimax estimators, a lot of space is dedicated
 to adaptive procedures. Adaptivity may be understood as minimax-adaptivity, i.e. optimal rates of convergence are attained simultaneously over a collection of class of densities, for example, over a collection of Sobolev-balls. Adaptivity may also refer to proving non-asymptotic oracle bounds, i.e. having a procedure that mimics, up to a constant, the estimator that minimizes a given loss function. It is this last notion of adaptivity we adopt in this article.

  Achieving adaptivity is a particular model selection issue for which there exist nume\-rous techniques.  Hereafter we mention some of them together with a non exhaustive list of references. Loosely speaking there exist three main approaches; thresholding techniques for wavelet density estimators (see e.g. \cite{donoho1994ideal,donoho1995wavelet,reynaud2011adaptive}), penalized estimators (see e.g. \cite{birge1998minimum,barron1999risk,massart2007concentration,lerasle2012optimal}) and pair wise comparison of estimators such as the Goldenshluger and Lepskii's procedure (see e.g. \cite{lepskii1991problem,lepskii1992asymptotically,goldenshluger2008universal,goldenshluger2011bandwidth,goldenshluger2013general,goldenshluger2014adaptive}). These techniques have been developed in a wide variety of contexts (for instance inverse problems or anisotropic multidimensional density estimation). Recently Lacour et al. \cite{lacour2017estimator} (see also the references therein) have introduced a new adaptive procedure for kernel density estimators, which is a modification of the Goldenshluger and Lepskii's method that has similar theoretical properties but is numerically more efficient.
 All the afore mentioned methods rely on the choice of an hyper parameter that needs to be calibrated, as explained in \cite{lacour2017estimator} numerical performances of the selected estimator  are ``very sensitive to this choice'' and  many studies have been devoted to the calibration of this hyper parameter (see e.g. Baudry et al. \cite{baudry2012slope}).
 Another technique that is popular, especially in numerical studies, is cross validation, which does not rely on the calibration of an hyper parameter. However cross validation does not permit to obtain theoretical results in general, is time-consuming in practice and in  some cases has very poor performances as explained in Delaigle and Gibels \cite{delaigle2004practical}.\\

In the present paper, we  put the stress on an adaptive method that was successfully applied in Duval and Kappus \cite{duval2017nonparametric} and whose scope goes beyond the grouped data setting studied therein. The adaptive procedure presented below does not outperform existing techniques, it suffers a logarithmic loss, but has the advantage of being numerically simple and fast and permits to compute the optimal cutoff parameter for different classical inverse pro\-blems. This method also depends on the calibration of an hyper parameter, on the numerical study it seems that the method shows little sensibility to this parameter. To illustrate it both theoretically and numerically, we focus on two classical inverse problems: deconvolution, the density estimation problem is a particular case, and decompounding. These issues have  raised a great interest in the literature (see the numerous references listed in Sections \ref{sec:exD} and \ref{sec:exP}). For those two inverse problems we provide adaptive estimators that minimizes, up to a multiplicative constant, the upper bound which brings optimal rates of convergence.

In the deconvolution problem, one observes $n$ i.i.d. realizations of $Y_{i}=X_{i}+\eps_{i}$, where $X_{i}$ and $\eps_{i}$ are independent, the density of $\eps_{1}$ is known and the density $f$ of $X_{1}$ is the quantity of interest.  Our study relies on a well studied optimal Fourier estimator and for which the optimal upper bound for the $\lk^{2}$-risk is known. Starting from there, we make explicit our cutoff parameter in this context and show that its $\lk^{2}$-risk minimizes the optimal upper bound up to a logarithmic factor. We conduct an extensive simulation study which illustrates the stability of the procedure and we compare our results with a penalization procedure for which many results have been developed in this context.

In the decompounding problem, one discretely observes one trajectory of a compound Poisson process $Z$,
\[Z_{t}=\sum_{i=j}^{N_{t}}X_{j},\quad t\geq 0,\] where $N$ is a Poisson process with intensity $\lambda$ independent of the i.i.d. random variables $(X_{j})_{j\in\N}$ with common density $f$. Let $\Delta>0$, suppose we observe $(Z_{i\Delta},\ i=1,\ldots,n)$, we aim at estimating $f$ from these observations. The cases $\Delta\to 0$ (high frequency observations) or $\Delta$ being fixed (low frequency observations, often $\Delta=1$) have been broadly studied in the literature (see the references given in Section \ref{sec:exP}) but were considered as separate cases. The case where $\Delta$ grows to infinity has never been studied. Therefore, we first study the problem of estimating the jump density $f$ of a compound Poisson process $Z$ from $(Z_{i\Delta},\ i=1,\ldots,n)$ for general sampling rates $\Delta$. We invert the L\'evy-Kintchine formula, relating the characteristic function of $Z_{\Delta}$ to the one of $X_{1}$. This approach is classical in the decompounding literature. We establish an upper bound for its $\lk^2$-risk where the dependency in $\Delta$ is made explicit. This upper bound is optimal simultaneously for $\Delta:=\Delta_{n}\to\Delta_{0}\in[0,\infty)$, it is also valid for $\Delta_{n}\to\infty$ under additional constraints (see Section \ref{sec:exP}). This result is new and presents interest on its own. The dependency in $\Delta_{0}$ of the upper bound shows a deterioration as $\Delta_{0}$ increases, which is expected. But, we identify regimes where $\Delta_{n}\to\infty$ such that $\Delta_{n}<1/4\log(n\Delta_{n})$ as $n\to\infty$ and where the estimator remains consistent, and presumably rate optimal.  Heuristically, if the jump density has finite variance, for simplicity assume $f$ is centered with unit variance, then, the law of each increment can be approximated as follows $X_{\Delta}=\sqrt{\lambda\Delta}\zeta_{\Delta},$ where $\zeta_{\Delta}\to\mathcal{N}(0,1)$ as $\Delta\to\infty$.   Therefore, one would expect that in these regimes  non-parametric estimation is impossible as each increment is close, in law, to a parametric Gaussian variable.  When $\Delta$ goes too rapidly to infinity, namely as a power of $n\Delta_{n}$, Duval \cite{duval2014no} shows that consistent non-parametric estimation of $f$ is impossible, regardless of the choice of the loss function, by showing that it is always possible to build two different compound Poisson processes with different jump densities for which the statistical experiments generated by their increments are asymptotically equivalent. The results of the present paper complement the knowledge on decompounding. Finally, we show that our adaptive procedure leads to an adaptive and rate optimal estimator of the jump density $f$, up to a logarithmic loss, for all sampling rates such that $\Delta_{n}<1/4\log(n\Delta_{n})$ as $n\to\infty$, this condition is fulfilled for fixed or vanishing $\Delta$. \\

 The article is organized as follows. In the remaining of this Section we describe the  idea of our adaptive procedure. In Section \ref{sec:exD} we illustrate, both theoretically and numerically, its performances for the deconvolution problem. In the numerical part we compare our procedure with a penalized adaptive optimal estimator. Section \ref{sec:exP} is dedicated to the decompounding problem. Finally, Section \ref{sec:proof} gathers the proofs of the main results.

\subsection{Methodology}

\paragraph{Notations.}
We introduce  some notations which are used throughout the rest of the text. Given a random variable $Z$,
$\phi_Z(u)=\E[e^{iu Z} ]$ denotes the characteristic function of $Z$. For $f\in \lk^1(\R)$, $\Fourier f(u) =\int 
e^{iux } f(x) \d x$ is understood to be the Fourier transform of $f$.  Moreover, we denote by $\| \cdot  \|$ the $\lk^2$-norm of functions, $\| f\|^2:= \int |f(x)|^2 \d x$.  Given some function $f\in \lk^1(\R) \cap \lk^2(\R)$, we denote by $f_m$ the  uniquely defined function with Fourier transform $\Fourier f_m= (\Fourier f)\mathds{1}_{[-m,m]}$.

\paragraph{Statistical setting.}

Consider $n$ i.i.d. realizations $Y_j,\, j=1,...,n$ of a random variable $Y$ with Lebesgue-density  $f_{Y}$. Suppose $Y$ is related to a variable $X$, with Lebesgue-density $f$ through a known transformation $\mathbf{T}$ relating their characteristic functions: $\phi_{Y}=\mathbf{T}(\phi_{X})$, where $\mathbf{T}$ can be linear (e.g. deconvolution) or not (grouped data (see \cite{meister2007optimal,duval2017nonparametric}), decompounding). We are interested in estimating the density $f$ of $X$ from the $(Y_j)$. 

We focus on estimators based on Fourier methods, which is convenient for several classes of inverse problems as the convolution operation corresponds to multiplication.
If the transformation $\mathbf{T}$ admits a continuous inverse, we build an estimator $\hat\phi_{X,n}$ of  $\phi_{X}$ from the observations $(Y_{1},\ldots,Y_{n})$,
\[\hat\phi_{X,n}(u)=\mathbf{T}^{-1}\big(\hat\phi_{Y,n}(u)\big)\quad\mbox{where}\ \hat\phi_{Y,n}(u):=\frac{1}{n}\sum_{j=1}^{n}e^{iuY_{j}},\ u\in\R.\]
Cutting off in the spectral domain and applying Fourier inversion gives  an estimator of $f$
\[
\hat f_{m}(x)=\frac{1}{2\pi}\int_{- m}^{ m}e^{-iux}\hat\phi_{X,n}(u)\d u,\quad \forall m>0,\ x\in\R.
\]
The performance of the estimator is measured with $\lk^2$-loss. The choice of the cutoff parameter is crucial: the goal is to select $m$ that mimics the optimal cutoff $m^{\star}$ which minimizes the $\lk^{2}$-risk,
\begin{align}\label{eq:mUB}
\E\big[\|\hat f_{m^{\star}}-f\|^{2}\big]=\underset{m\geq 0}{\inf}\Big\{\frac{1}{2\pi}\int_{[-m,m]^{c}}|\phi_{X}(u)|^{2}\d u+\frac{1}{2\pi}\int_{-m}^{m}\E\big[|\hat \phi_{X,n}(u)-\phi_{X}(u)|^{2}\big]\d u\Big\}.\end{align}
 This optimal value $m^{*}$ usually depends on the unknown regularity of $f$ and is hence not feasible. We propose a procedure to select a  random cutoff $\hat m_{n}$,  which can be calculated from the observations, and for which  the  $\lk^2$-risk is close to the one of $\hat f_{m^{*}}$, meaning that we can establish an oracle bound
\[\E\big[\|\hat f_{\hat m_{n}}-f\|^{2}\big]\leq C \underset{m\geq 0}{\inf}\Big\{\frac{1}{2\pi}\int_{[-m,m]^{c}}|\phi_{X}(u)|^{2}\d u+\frac{1}{2\pi}\int_{-m}^{m}\E\big[|\hat \phi_{X,n}(u)-\phi_{X}(u)|^{2}\big]\d u\Big\}+r_{n},\] for a positive constant $C$ and $r_{n}$ a negligible remainder. We call  $\hat f_{\hat m_{n}}$ adaptive rate optimal estimator of $f$.

 \paragraph{Heuristic of the adaptive procedure.} If $\mathbf{T}^{-1}$ is differentiable and if we can show that for some positive constant $C$, $\E\big[|\hat \phi_{X,n}(u)-\phi_{X}(u)|^{2}\big]\leq C |(\mathbf{T}^{-1})^{\prime}\big(\phi_{Y}(u)\big)|^{2}\E\big[|\hat \phi_{Y,n}(u)-\phi_{Y}(u)|^{2}\big]$, we have, 
 \begin{align}\label{eq:UBintro}
\E\big[\|\hat f_{m}-f\|^{2}\big]\leq \frac{1}{2\pi}\int_{[-m,m]^{c}}|\phi_{X}(u)|^{2}\d u+\frac{C}{n}\int_{-m}^{m}\big|(\mathbf{T}^{-1})^{\prime}\big(\phi_{Y}(u)\big)\big|^{2}\d u,\quad m\geq 0.\end{align}
  The quantity $(\mathbf{T}^{-1})^{\prime}$ is explicit in the grouped data setting (see \cite{meister2007optimal,duval2017nonparametric}), but also in the deconvolution and decompounding cases (see e.g. Sections \ref{sec:exD} and \ref{sec:exP} hereafter). The second term in the right hand side of the latter inequality is a majorant of the integrated variance of the estimator; using a majorant of the variance term is the starting point of many adaptive procedures such as penalized procedures or the Goldenshluger and Lepskii's procedure, where from this majorant, one tries to find a cutoff such that the empirical bias and the majorant are of the same order. 

Denote by $\overline{m}_{n}$ the arginf of the right hand side of \eqref{eq:UBintro}. If the upper bound \eqref{eq:UBintro} is optimal, meaning that it has the same order as \eqref{eq:mUB}, then asymptotically it holds that $m^{\star}\asymp \overline{m}_{n}.$ Differentiating in $m$ the right hand side of \eqref{eq:UBintro} gives that the cutoff $\overline{m}_{n}$ that minimizes this quantity satisfies
\begin{align}\label{eq:mheur}|\phi_{X}(\overline{m}_{n})|^{2}=\frac{C}{n}\big|(\mathbf{T}^{-1})^{\prime}\big(\phi_{Y}(\overline{m}_{n})\big)\big|^{2}\Longleftrightarrow|\mathbf{T}(\phi_{Y}(\overline{m}_{n}))|^{2}=\frac{C}{n}\big|(\mathbf{T}^{-1})^{\prime}\big(\phi_{Y}(\overline{m}_{n})\big)\big|^{2}.\end{align}Clearly, \eqref{eq:mheur} has an empirical version and it is tempting to select $\hat{m}_{n}$ accordingly. This inspires to select the cutoff parameter in the following ensemble, 
\begin{align}\label{eq:mheuremp}\hat m_{n}\in \Big\{m>0, \Big|\frac{\mathbf{T}(\hat\phi_{Y,n}(m))}{\big(\mathbf{T}^{-1})^{\prime}(\hat\phi_{Y,n}(m))}\Big|= \frac{1}{\sqrt{n}}\big(1+\kappa\sqrt{\log n}\big)\Big\}\wedge n,\end{align}
 for some $\kappa>0$. However, the solution of \eqref{eq:mheur} may not be unique; but considering the minimum (as in \cite{duval2017nonparametric} where the density of interest also plays the role of the noise) or the maximum of this ensemble (as in the sequel), it is uniquely determined. Many adaptive procedures such as penalization methods minimizes an empirical version of the upper bound \eqref{eq:UBintro} whereas the spirit of \eqref{eq:mheuremp} consists in finding the zeroes of an empirical version of the derivative in $m$  of the upper bound \eqref{eq:UBintro}. Roughly speaking, the difference between our procedure and a penalization procedure is the same as the difference between Z-estimators and M-estimators.

The quantity involved in the definition of $\hat m_{n}$ also appears in the definition of the estimator. This explains why the procedure  performs numerically fast. It has proven to be asymptotically minimax (up to a logarithmic loss) in the grouped data setting and we show that it also leads to adaptive rate optimal estimators  (up to a logarithmic loss) in deconvolution and decompounding inverse problems.

\section{Deconvolution\label{sec:exD}}

\subsection{Statistical setting}

Suppose that $X_1,\dots, X_n$ are i.i.d. with density $f$ and are accessible through the noisy observations 
\[
Y_j = X_j +\eps_j, \ j=1, \ldots, n.
\]
Assume that the  $(\eps_j)$ are i.i.d., independent of the $(X_j)$ and such  that $\forall u\in \R,\ \phi_{\eps}(u)\ne 0$.  Suppose that the distribution of $\eps_1$ is known. This last assumption can be softened, the procedure allows a straightforward generalization to the case where the distribution of $\eps_1$ can be estimated from an additional sample, see Neumann \cite{MR1460203}.

A deconvolution estimator of the characteristic function $\phi_X$ of $X$ is given by 
\begin{align*}
\frac{\hat\phi_{Y,n}(u)}{\phi_{\eps}(u)},\quad \mbox{with }\ \hat\phi_{Y,n}(u):=\frac1n {\sum_{j=1}^{n} e^{iuY_j}  } ,\quad u\in\R, 
\end{align*}
denoting the empirical characteristic function.  Since $\phi_X$ is a characteristic function, its absolute value is bounded by 1 and the estimator can hence be improved by using the definition 
\begin{align}\label{eq:hatphiD}
\hat{\phi}_{X,n}(u): =\frac{\hat\phi_{Y,n}(u)}{\phi_{\eps}(u)} \frac1 { \max\{1,\big |\frac{\hat\phi_{Y,n}(u)}{\phi_{\eps}(u)}\big|\}  },\quad u\in\R.
\end{align}
Cutting off in the spectral domain and applying Fourier inversion gives  the estimator
\begin{align}\label{eq:fhatD}
\hat{f}_m(x)  = \frac{1}{2\pi} \int_{-m}^{m}e^{- iux} \hat{\phi}_{X,n}(u) \d u,\quad x\in\R. 
\end{align} 
This estimator and adaptation techniques have been extensively studied in the literature, including  in more general settings than above. Optimal rates of convergence and adaptive procedures are well known if $d=1$ (see {e.g.}  \cite{carroll1988optimal,stefanski1990rates,stefanski1990deconvolving,fan1991optimal,butucea2004deconvolution,MR2354572,MR2742504,pensky1999adaptive,MR2416478} for $\lk^2$-loss functions or \cite{lounici2011uniform} for the $\lk^{\infty}$-loss). Results have also been established for multivariate anisotropic densities (see {e.g.} \cite{comte2013anisotropic} for $\lk^2$-loss functions or \cite{rebelles2016structural} for $\lk^{p}$-loss functions, $p\in [1,\infty]$). Deconvolution with unknown error distribution has also been studied (see {e.g.}  \cite{MR1460203,MR2396811,johannes2009deconvolution,zbMATH05358925}, if an additional error sample is available, or \cite{comte2011data,delattre2012blockwise,johannes2013adaptive,comte2015density,kappus2014adaptive} under other set of assumptions).

\subsection{Risk bounds and adaptive bandwidth selection}
The following risk bound is well known in the literature on deconvolution. 
\begin{align}
\hspace{-0.3cm}\E\big[\|\hat f_{m}-f\|^{2}\big]\leq&  \|f-f_{m}\|^{2} +  \frac{1}{2\pi n}  \int_{-m}^{m}  \limits
\frac{  \d u}{|\phi_\eps(u)|^2}=\frac{1}{2\pi}\Big(\hspace{-0.3cm}\int_{[-m,m]^{c}}\limits\hspace{-0.3cm}|\phi_{X}(u)|^{2}\d u+\frac{1}{n}\int_{-m}^{m}\limits\frac{\d u}{|\phi_{\eps}(u)|^{2}}\Big). \label{eq:UBD}
\end{align}This upper bound is the sum of a bias term that decreases with $m$ and a variance term, increasing with $m$.
We select $\overline m_{n}$, the optimal cutoff parameter, such that the upper bound \eqref{eq:UBD} is minimal
\[\overline m_{n}\in\underset{m\geq0}{\mbox{arginf}}\Big\{\hspace{-0.3cm}\int_{[-m,m]^{c}}\limits\hspace{-0.3cm}|\phi_{X}(u)|^{2}\d u+\frac{1}{n}\int_{-m}^{m}\limits\frac{\d u}{|\phi_{\eps}(u)|^{2}}\Big\}.
\]
At $\overline m_{n}$ both terms in \eqref{eq:UBD} are of the same order which gives the optimal cutoff. Differentiating the right hand side with respect to  $m$, we find that the following holds for the optimal cutoff parameter:
\begin{align}\label{eq:equalD}
|\phi_{X}(\overline m_{n})|^{2}= \frac{1}{n|\phi_{\eps}(\overline m_{n} )|^{2}}\quad \Longleftrightarrow\quad |\phi_{X}(\overline m_{n})\phi_{\eps}(\overline m_{n})|=|\phi_{Y}(\overline m_{n})|=\frac{1}{\sqrt{ n}}.
\end{align}
This equality has an empirical version and we select $\hat{m}_{n}$ accordingly.  In order to ensure adaptivity the following heuristic consideration is helpful. When the characteristic function is replaced by its empirical version, the standard deviation is of the order $n^{-1/2}$.  Consequently, estimating $\phi_Y$ by $\hat{\phi}_{Y,n}$ makes sense for $|\phi_Y|\geq n^{-1/2}$.  If $|\phi_Y|<n^{-1/2}$, the noise is dominant so the estimator might be set to zero. 
This inspires to re-define the estimator of $\phi_Y$ as follows: 
\[
\tilde{\phi}_{Y,n}(u)=\hat{\phi}_{Y,n}(u) \mathds{1}_{\{|\hat{\phi}_{Y,n}(u)|\geq {\kappa}_n n^{-1/2}\}},\quad u\in\R,
\]
with the threshold value ${\kappa}_n:= (1 +\kappa \sqrt{\log n}  )$. The constant $\kappa>0$ is specified below.  Then, the estimator of $f$ given in \eqref{eq:fhatD} is modified using the following instead of \eqref{eq:hatphiD}\[
\tilde{\phi}_{X,n}(u): =\frac{\tilde\phi_{Y,n}(u)}{\phi_{\eps}(u)} \frac1 { \max\{1,\big |\frac{\tilde\phi_{Y,n}(u)}{\phi_{\eps}(u)}\big|\}  },\quad u\in\R
\]  and \[
\tilde{f}_m(x)  = \frac{1}{2\pi} \int_{-m}^{m}e^{- iux} \tilde{\phi}_{X,n}(u) \d u,\quad x\in\R. 
\] Note that the upper bound \eqref{eq:UBD} remains valid for the estimator $\tilde f_{m}$, thanks to the result:
\begin{lemma}
\label{lem:tildehat}Let $z=re^{i\theta}$, with $r\leq 1$, $\hat z=\rho e^{iw}$, $\rho> 1$, and $\tilde z=e^{iw}$. Then, $|\tilde z-z|\leq|\hat z-z|.$
\end{lemma} We define the empirical cutoff parameter $\hat m_{n}$ as follows. Since $\hat \phi_{Y,n}$ may show an oscillatory behavior and the solution of \eqref{eq:equalD} may not be unique, we consider 
\begin{align}\label{eq:mD}
\hat{m}_{n}=\max\Big\{m>0: |\hat{\phi}_{Y,n}(m) |  ={\kappa}_n n^{-1/2}\Big\}\wedge n^\alpha,
\end{align}
for some $\alpha\in(0,1]$. It is worth emphasizing that the calculation of $\hat{m}_{n}$ does only rely on the empirical characteristic function $\hat{\phi}_{Y,n}$ and does  not require the evaluation of penalty terms  depending  on the (perhaps unknown) $\phi_\eps$.

\begin{theorem}\label{thm:AD}
Let $\hat m_{n}$ defined as in \eqref{eq:mD}, with $ \kappa > 0$ and $\alpha\in(0,1]$. Then, 
there exist a positive constant $C_1$ depending only on the choice of $\kappa$ and a universal positive constants $C_2$  such that 
\[
\E[  \|f - \tilde{f}_{\hat m_{n}}\|^2 ]  \leq C_1\underset{m\in[0,n^{\alpha}]}{\inf}\Big\{\int_{[-m,m]^{c}}\limits\hspace{-0.3cm}|\phi_{X}(u)|^{2}\d u+\frac{\log n}{n}\int_{-m}^{m}\frac{\d u}{|\phi_{\eps}(u)|^{2}}\Big\}+ C_2 n^{\alpha-\kappa^{2}/2 }. 
\]
\end{theorem}
Theorem \ref{thm:AD} is non asymptotic and ensures that the estimator $\tilde f_{\hat m_{n}}$ automatically reaches the bias-variance compromise, up to a logarithmic factor and the multiplicative constant $C_{1}$. The logarithmic loss is technical; it permits to control the deviation of the empirical characteristic functions from the true characteristic function by the Hoeffding inequality. Without changing the strategy of the proof removing the additional log factor seems difficult. Proof of Theorem \ref{thm:AD} is self contained.  

 Regarding the choice of $\alpha$ and $\kappa$ in \eqref{eq:mD}, it is always possible to take $\alpha=1$. Note that the case $\alpha>1$ is not interesting as, even in the direct problem $\eps=0$ a.s.,  if $m>n$ the variance term in \eqref{eq:UBD} no longer tends to 0. Taking $\alpha<1$ is possible only if one has additional information on the target density $f$. For instance, if one knowns that $f$ is in a Sobolev class of regularity $\beta$, for some $\beta\geq\beta_{0}>0$,
\begin{align}\label{eq:Sob}f\in \mathscr{S}(\beta,L):=\Big\{f\in \textbf{F},\ \int_{\R}(1+|u|)^{2\beta}|\Fourier{f}(u)|^{2}\d u\leq L\Big\}
\end{align}
where $\textbf{F}$ is the set of densities with respect to the Lebesgue measure. Then, it holds that $\|f-f_{m}\|^{2}\asymp m^{-2\beta}$ and straightforward computations lead to $m^{\star}\lesssim \big(\frac n{\log n}\big)^{\frac{1}{2\beta+1}}$ (regardless the the asymptotic decay of $\phi_{\eps}$). Then, one may restrict the interval for $\hat m_{n}$ to $[0,n^{\alpha}]$ where $1>\alpha>\frac{1}{2\beta_{0}+1}$. Second, the choice of $\kappa$ must be such that $n^{\alpha-\kappa^{2}/2}$ is negligible, the choice $\kappa>\sqrt{2}$ always works. The following numerical study suggests that the procedure is stable in the choice of $\kappa$.

\subsection{Numerical results.}
\paragraph{Stability of the procedure.}
To illustrate the performances of the method and the influence of the parameter $\kappa$  we proceed as follows. For different densities $f$, namely, Uniform $\mathcal{U}[1,3]$, Gaussian $\mathcal{N}(2,1)$, Cauchy, Gamma $\Gamma(2,1)$ and the mixture $0.7\mathcal{N}(4,1)+0.3 \Gamma(2,\frac{1}{2})$, and for different values of $\kappa$ we compute  the adaptive $\lk^2$ risks from $M=1000$  Monte Carlo iterations. The results are displayed on Figures \ref{Fig:direct}, \ref{Fig:DS} and \ref{Fig:DSS}.
We consider three different settings:
\begin{itemize}
\item
{The direct  density estimation problem} (Figure \ref{Fig:direct}): we observe i.i.d. realizations of $f$. It is a particular deconvolution problem where $\eps=0$ a.s. 
\item
{Deconvolution problem with ordinary smooth noise} (Figure \ref{Fig:DS}): the error $\eps$  is Gamma $\Gamma(2,1)$ i.e. $|\phi_{\eps}|$ decays as $|u|^{-2}$ asymptotically.
\item
{Deconvolution problem with super smooth noise} (Figure \ref{Fig:DSS}): the error $\eps$  is Cauchy i.e. $|\phi_{\eps}|$ decays as $e^{-|u|}$ asymptotically.\end{itemize}

On Figures \ref{Fig:direct}, \ref{Fig:DS} and \ref{Fig:DSS} we observe that the adaptive rates are small and that the procedure is stable on the choice of kappa.
We observe on these three cases, that the value of $\kappa$ should not be chosen too large but that for a wide range of values the performances are similar. In practice, the value of $n$ is fixed and there is a natural boundary for $\kappa$, indeed observe that it is useless to increase $\kappa$ if $(1+\kappa\log n)n^{-1/2}\geq 1$ as the selection rule \eqref{eq:mD} will be constant equal to $n^{\alpha}$. Moreover, we expect that if $(1+\kappa\log n)n^{-1/2}$ gets too large, e.g. larger than 1/2 the performances of the adaptive estimator should deteriorate. This practical consideration encourages to choose $\kappa$ smaller than  $({\sqrt{n}-1}){\sqrt{\log(n)}}^{-1}$. In Figures \ref{Fig:direct}, \ref{Fig:DS} and \ref{Fig:DSS}  it appears that for all the meaningful values of $\kappa$, e.g. smaller than $\frac12({\sqrt{n}-1}){\sqrt{\log(n)}}^{-1}$ for instance, the performances of the adaptive estimator are similar.

\begin{figure}[h!]\begin{center}
\includegraphics[width=10cm,height=5cm]{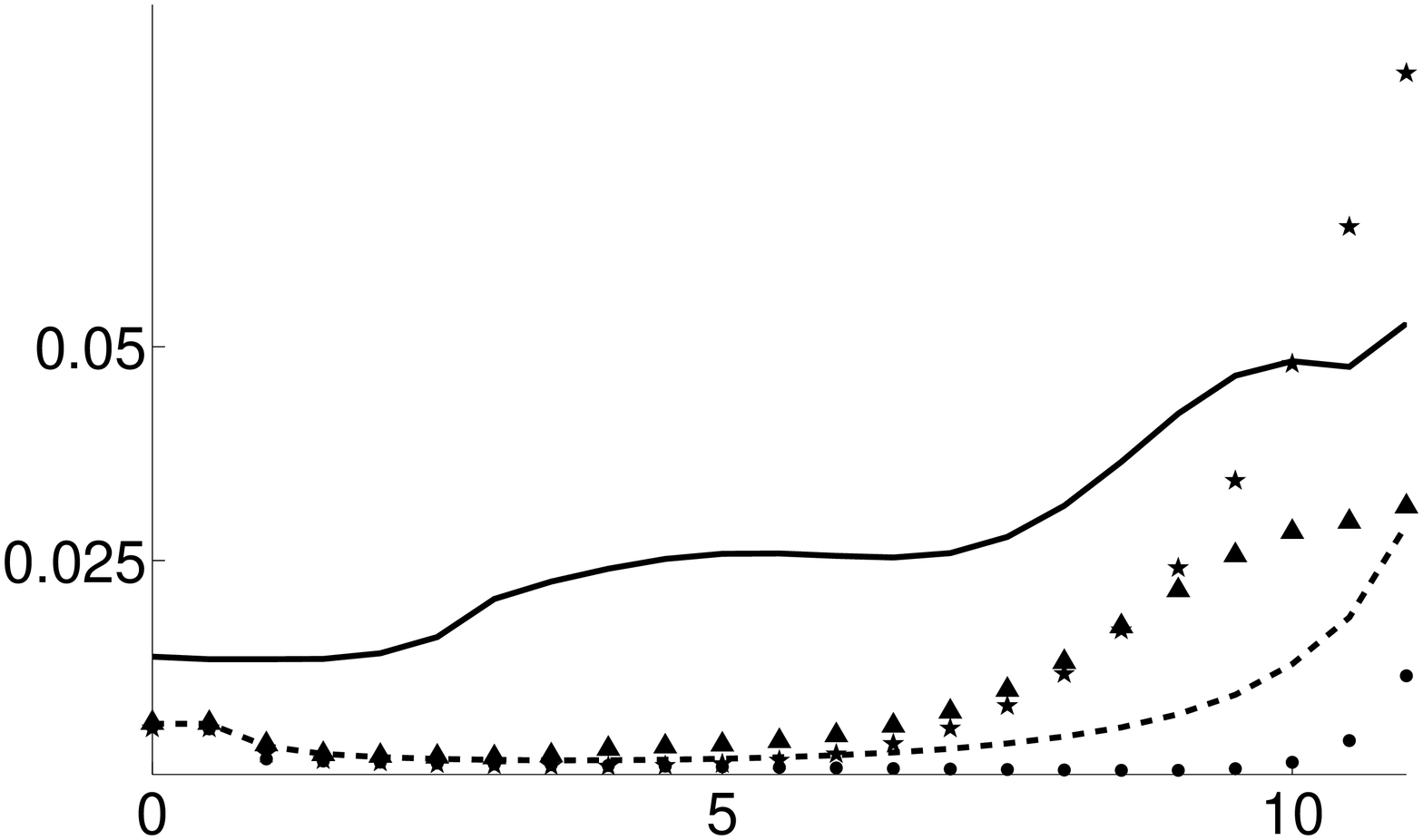}
\caption{\label{Fig:direct}\footnotesize{\textbf{Direct problem:} Computations by $M=1000$ Monte Carlo iterations of the $\lk^2$-risks ($y$ axis) for different values of $\kappa\leq({\sqrt{n}-1}){\sqrt{\log(n)}}^{-1}=11.6$ ($x$ axis). Estimation of $f$ from $n=1000$ i.i.d. direct realizations for different distributions: Uniform $\mathcal{U}[1,3]$ (plain line), Gaussian $\mathcal{N}(2,1)$ (dots), Cauchy (stars), Gamma $\Gamma(2,1)$ (dotted line) and the mixture $0.7\mathcal{N}(4,1)+0.3 \Gamma(2,\frac{1}{2})$ (triangles). }}\end{center}
\end{figure}

\begin{figure}[h!]\begin{center}
\includegraphics[width=10cm,height=5cm]{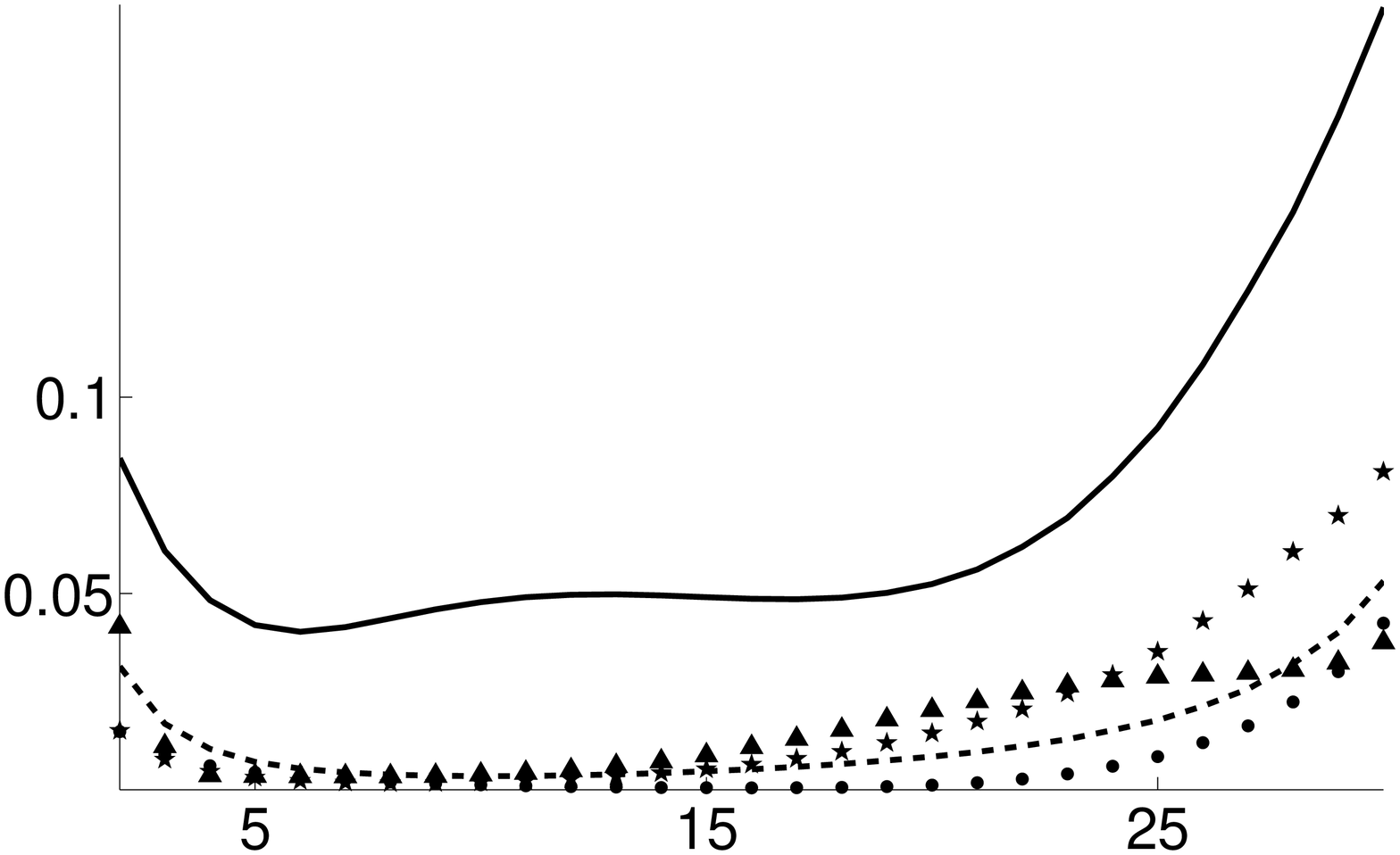}\end{center}
\caption{\label{Fig:DS}\footnotesize{\textbf{Deconvolution problem (ordinary smooth case):} Computations by $M=1000$ Monte Carlo iterations of the $\lk^2$-risks ($y$ axis) for different values of $\kappa\leq({\sqrt{n}-1}){\sqrt{\log(n)}}^{-1}=32.6$ ($x$ axis). Estimation of $f$ from $n=10000$ i.i.d. direct  of $X+\eps$ where $\eps$ has distribution $\Gamma(2,1)$ (i.e. $\phi_{\eps}(u)=(1-iu)^{-2}$)  and for different distributions for $X$: Uniform $\mathcal{U}[1,3]$ (plain line), Gaussian $\mathcal{N}(2,1)$ (dots), Cauchy (stars), Gamma $\Gamma(2,1)$ (dotted line) and the mixture $0.7\mathcal{N}(4,1)+0.3 \Gamma(2,\frac{1}{2})$ (triangles).  }}
\end{figure}

\begin{figure}[h!]\begin{center}
\includegraphics[width=10cm,height=5cm]{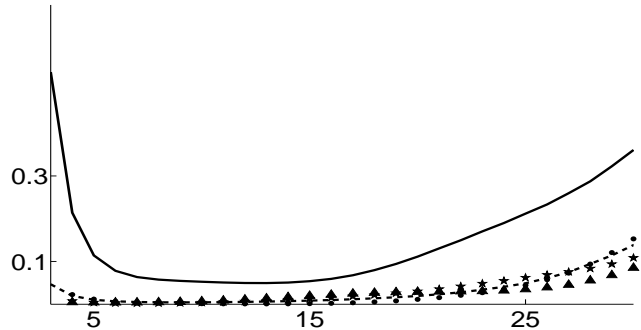}\end{center}
\caption{\label{Fig:DSS}\footnotesize{\textbf{Deconvolution problem (super smooth case):} Computations by $M=1000$ Monte Carlo iterations of the $\lk^2$-risks ($y$ axis) for different values of $\kappa\leq({\sqrt{n}-1}){\sqrt{\log(n)}}^{-1}=32.6$ ($x$ axis). Estimation of $f$ from $n=10000$ i.i.d. realizations of $X+\eps$ where $\eps$ has Cauchy distribution (i.e. $\phi_{\eps}(u)=e^{-|u|}$) and for different distributions for $X$: Uniform $\mathcal{U}[1,3]$ (plain line), Gaussian $\mathcal{N}(2,1)$ (dots), Cauchy (stars), Gamma $\Gamma(2,1)$ (dotted line) and the mixture $0.7\mathcal{N}(4,1)+0.3 \Gamma(2,\frac{1}{2})$ (triangles). For the uniform distribution, the rates where stable around the value 0.5, they do not appear on the Figure not to spoil the readability of the other curves.}}
\end{figure}

\paragraph{Comparison with other procedures.}
We compare the performances of our procedure for $\kappa=8$,  with a penalization procedure and with an oracle. For the penalization procedure, we follow Comte and Lacour \cite{comte2011data} and consider the adaptive estimator $\hat f_{\tilde m_{n}}$ which is the estimator defined in \eqref{eq:fhatD}  where 
\[\tilde m_{n}=\underset{m\in[0,M_{n}]}{\mbox{argmin}}\{-\|\hat f_{m}\|^{2}+\mbox{pen}(m)\},\ \mbox{pen}(m)=K\Big(\frac{\Delta(m)}{\log(m+1)}\Big)^{2}\frac{\Delta(m)}{n},\] where $M_{n}>0,$ $K>0$ and $\Delta(m)=\frac{1}{2\pi}\int_{[-m,m]}|\phi_{\eps}(u)|^{-2}\d u$, which is known in our setting. The  parameter $M_{n}$ is chosen as the maximal integer such that $1\leq\frac{\Delta(m)}{n}\leq 2$. For the parameter $K$ it is calibrated by preliminary simulation experiments. For calibration strategies (dimension jump and slope heuristics), the reader is referred to Baudry et al. \cite{baudry2012slope}. Here, we test a grid of values of the $K$’s from the empirical error point of view, to make a relevant choice; the tests are conducted on a set of densities which are different from the one considered hereafter, to avoid overfitting. After these preliminary experiments, $K$ is chosen equal to $2$ which is the same value as the one considered in Comte and Lacour \cite{comte2011data}. The standard errors are given in parenthesis. The running times for each risks of the penalization procedure and our procedure are similar; our procedure is barely faster.  However, one should take into account that a preliminary calibration step seems obsolete in our case. In deconvolution problems, the theoretical optimal $K$ can be in some cases far away from the practically optimal $K$ and may vary with the sample size explaining the nessecity of this calibration step (see e.g. Kappus and Mabon \cite{kappus2014adaptive} where the practical optimal value of $K$ was much smaller than the value predicted by the theory).

Second, an oracle ”estimator” is computed $\hat f_{m^{\star}}$, which is the estimator defined in \eqref{eq:fhatD} where $m^{\star}$ corresponds to the following oracle bandwidth
\[m^{\star} =\underset{m>0}{\mbox{argmin }}\E[\|f-\hat f_{m}\|^{2}].\]
This oracle can be explicitly evaluated when $f$ is known. We denote these different risks by $R$, for the risk of our procedure, $R_{pen}$ for the penalized estimator and $R_{or}$ for the oracle procedure. All these risks are computed on $1000$ Monte Carlo iterations. The results are gathered in Tables \ref{Table1} for the Gamma density, \ref{Table2} for the mixture  and \ref{Table3} for the Cauchy density where $\mathcal{C}$ stands for the Cauchy distribution. In each case both an ordinary smooth and a super smooth errors are considered.  \\

\begin{table}
\begin{center}
\begin{tabular}{c c cccccc}
\multirow{2}{*}{\diagbox{$f_{\eps}$}{$f$} }& & \multicolumn{6}{c}{$\Gamma(2,1)$} \\
 & $n$ & $R$  & $\hat m$ & $R_{pen}$ & $\tilde m$ & $R_{or}$  & $m^{\star}$\\
 \hline
 \hline
\multirow{6}{*}{$\Gamma(2,1)$}&  \multirow{2}{*}{ 500} & {\small $4.31\times 10^{-2}$ }   & {\small 1.05}&{\small $1.97\times 10^{-2}$ }&{\small 0.80} &{\small $0.74\times 10^{-2}$ }&{\small 0.66 }  \\
&   &  {\scriptsize{($0.20$)}}  & {\scriptsize{($0.07$)}}& {\scriptsize{($0.01$)}}&{\scriptsize{($0.05$)}}&{\scriptsize{($  0.36\times 10^{-2}$)}} &{\scriptsize{($0.14  $)}} \\
\cline{2-8}
& \multirow{2}{*}{ 1000} & {\small $1.74\times 10^{-2}$ }   & {\small 0.98}&{\small $1.70\times 10^{-2}$ }&{\small 0.94} &{\small $0.59\times 10^{-2}$ }&{\small 0.72}  \\
  & &  {\scriptsize{($0.13$)}}  & {\scriptsize{($0.04$)}}& {\scriptsize{($0.03$)}}&{\scriptsize{($0.03  $)}}&{\scriptsize{($0.28\times 10^{-2}  $)}} & {\scriptsize{($0.14  $)}} \\
  \cline{2-8}
  &  \multirow{2}{*}{ 5000} & {\small $0.40\times 10^{-2}$ }   & {\small 0.85}&{\small $1.30\times 10^{-2}$ }&{\small 1.32} &{\small $0.31\times 10^{-2}$ }&{\small 0.91 }  \\
&   &  {\scriptsize{($0.06$)}}  & {\scriptsize{($0.01$)}}& {\scriptsize{($0.01$)}}&{\scriptsize{($0.05$)}}&{\scriptsize{($  0.13\times 10^{-2}$)}} &{\scriptsize{($0.15  $)}} \\
\hline
\hline
\multirow{6}{*}{$\mathcal{C}$}&  \multirow{2}{*}{ 500} & {\small $5.27\times 10^{-2}$ }   & {\small 0.90}&{\small $1.21\times 10^{-2}$ }&{\small 0.56 } &{\small $0.92\times 10^{-2}$ }&{\small 0.61}  \\
&   &  {\scriptsize{($0.23$)}}  & {\scriptsize{($0.07$)}}& {\scriptsize{($0.69\times 10^{-2}$)}}&{\scriptsize{($  0.04$)}}&{\scriptsize{($  0.48\times 10^{-2}$)}} & {\scriptsize{($ 0.12$)}} \\
\cline{2-8}
& \multirow{2}{*}{ 1000} & {\small $1.84\times 10^{-2}$ }   & {\small 0.84}&{\small $0.98\times 10^{-2}$ }&{\small 0.70 } &{\small $0.70\times 10^{-2}$ }&{\small 0.67}  \\
  & &  {\scriptsize{($0.13$)}}  & {\scriptsize{($0.04$)}}& {\scriptsize{($0.61\times 10^{-2}$)}}&{\scriptsize{($0.03 $)}}&{\scriptsize{($ 0.34\times 10^{-2} $)}} &{\scriptsize{($0.13$)}} \\
\cline{2-8}
& \multirow{2}{*}{ 5000} & {\small $0.51\times 10^{-2}$ }   & {\small 0.70}&{\small $0.71\times 10^{-2}$ }&{\small 1.10 } &{\small $0.39\times 10^{-2}$ }&{\small 0.82}  \\
  & &  {\scriptsize{($0.07$)}}  & {\scriptsize{($0.01$)}}& {\scriptsize{($0.01\times 10^{-2}$)}}&{\scriptsize{($0.02 $)}}&{\scriptsize{($ 0.17\times 10^{-2} $)}} &{\scriptsize{($0.13$)}} \\
\hline
\end{tabular}\caption{\label{Table1} \footnotesize{Comparaison of the different adaptive estimators for the Gamma distribution.}}
\end{center}

\end{table}

\begin{table}
\begin{center}
\begin{tabular}{c c cccccc}
\multirow{2}{*}{\diagbox{$f_{\eps}$}{$f$} }& & \multicolumn{6}{c}{$0.7\mathcal{N}(4,1)+0.3\Gamma(4,\frac12)$} \\
 & $n$ & $R$  & $\hat m$ & $R_{pen}$ & $\tilde m$ & $R_{or}$  & $m^{\star}$\\
 \hline
 \hline
\multirow{6}{*}{$\Gamma(2,1)$}&  \multirow{2}{*}{ 500} & {\small $1.77\times 10^{-2}$ }   & {\small 0.92}&{\small $0.78\times 10^{-2}$ }&{\small 0.78} &{\small $0.33\times 10^{-2}$ }&{\small 0.59 }  \\
&   &  {\scriptsize{($0.13$)}}  & {\scriptsize{($0.05$)}}& {\scriptsize{($0.65\times 10^{-2}$)}}&{\scriptsize{($0.03$)}}&{\scriptsize{($  0.17\times 10^{-2}$)}} &{\scriptsize{($0.16  $)}} \\
\cline{2-8}
& \multirow{2}{*}{ 1000} & {\small $0.74\times 10^{-2}$ }   & {\small 0.89}&{\small $0.73\times 10^{-2}$ }&{\small 0.87} &{\small $0.26\times 10^{-2}$ }&{\small 0.67}  \\
  & &  {\scriptsize{($0.08$)}}  & {\scriptsize{($0.03$)}}& {\scriptsize{($0.64\times 10^{-2}$)}}&{\scriptsize{($0.03 $)}}&{\scriptsize{($0.15\times 10^{-2}  $)}} & {\scriptsize{($0.14  $)}} \\
  \cline{2-8}
& \multirow{2}{*}{ 5000} & {\small $0.13\times 10^{-2}$ }   & {\small 0.79}&{\small $0.38\times 10^{-2}$ }&{\small 1.10} &{\small $0.01\times 10^{-2}$ }&{\small 0.82}  \\
  & &  {\scriptsize{($0.03$)}}  & {\scriptsize{($0.01$)}}& {\scriptsize{($0.30\times 10^{-2}$)}}&{\scriptsize{($0.01 $)}}&{\scriptsize{($0.06\times 10^{-3}  $)}} & {\scriptsize{($0.11  $)}} \\
\hline
\hline
\multirow{6}{*}{$\mathcal{C}$}&  \multirow{2}{*}{ 500} & {\small $2.78\times 10^{-2}$ }   & {\small 0.82}&{\small $0.73\times 10^{-2}$ }&{\small 0.55 } &{\small $0.43\times 10^{-2}$ }&{\small 0.50}  \\
&   &  {\scriptsize{($0.15$)}}  & {\scriptsize{($0.05$)}}& {\scriptsize{($0.50\times 10^{-2}$)}}&{\scriptsize{($ 0.05$)}}&{\scriptsize{($  0.20\times 10^{-2}$)}} & {\scriptsize{($ 0.14$)}} \\
\cline{2-8}
& \multirow{2}{*}{ 1000} & {\small $1.02\times 10^{-2}$ }   & {\small 0.77}&{\small $0.65\times 10^{-2}$ }&{\small 0.67 } &{\small $0.34\times 10^{-2}$ }&{\small 0.58}  \\
  & &  {\scriptsize{($0.10$)}}  & {\scriptsize{($0.03$)}}& {\scriptsize{($0.52\times 10^{-2}$)}}&{\scriptsize{($0.03$)}}&{\scriptsize{($ 0.16\times 10^{-2} $)}} &{\scriptsize{($0.15$)}} \\
  \cline{2-8}
& \multirow{2}{*}{ 5000} & {\small $0.21\times 10^{-2}$ }   & {\small 0.66}&{\small $0.59\times 10^{-2}$ }&{\small 0.94 } &{\small $0.16\times 10^{-2}$ }&{\small 0.74}  \\
  & &  {\scriptsize{($0.04$)}}  & {\scriptsize{($0.01$)}}& {\scriptsize{($0.48\times 10^{-2}$)}}&{\scriptsize{($0.02 $)}}&{\scriptsize{($ 0.10\times 10^{-2} $)}} &{\scriptsize{($0.11$)}} \\
\hline
\end{tabular}\caption{\label{Table2} \footnotesize{Comparaison of the different adaptive estimators on a mixture.}}
\end{center}

\end{table}

\begin{table}
\begin{center}
\begin{tabular}{c c cccccc}
\multirow{2}{*}{\diagbox{$f_{\eps}$}{$f$} }& & \multicolumn{6}{c}{$\mathcal{C}$} \\
 & $n$ & $R$  & $\hat m$ & $R_{pen}$ & $\tilde m$ & $R_{or}$  & $m^{\star}$\\
 \hline
 \hline
\multirow{6}{*}{$\Gamma(2,1)$}&  \multirow{2}{*}{ 500} & {\small $2.69\times 10^{-2}$ }   & {\small 0.59}&{\small $1.00\times 10^{-2}$ }&{\small 0.68} &{\small $0.67\times 10^{-2}$ }&{\small 0.62 }  \\
&   &  {\scriptsize{($0.16$)}}  & {\scriptsize{($0.07$)}}& {\scriptsize{($0.67\times 10^{-2}$)}}&{\scriptsize{($0.03$)}}&{\scriptsize{($  0.29\times 10^{-2}$)}} &{\scriptsize{($0.10  $)}} \\
\cline{2-8}
& \multirow{2}{*}{ 1000} & {\small $1.00\times 10^{-2}$ }   & {\small 0.84}&{\small $0.97\times 10^{-2}$ }&{\small 0.82} &{\small $0.49\times 10^{-2}$ }&{\small 0.67}  \\
  & &  {\scriptsize{($0.09$)}}  & {\scriptsize{($0.04$)}}& {\scriptsize{($0.70\times 10^{-2}$)}}&{\scriptsize{($0.05  $)}}&{\scriptsize{($0.21\times 10^{-2}  $)}} & {\scriptsize{($0.10  $)}} \\
\cline{2-8}
& \multirow{2}{*}{ 5000} & {\small $0.27\times 10^{-2}$ }   & {\small 0.69}&{\small $0.81\times 10^{-2}$ }&{\small 1.18} &{\small $0.21\times 10^{-2}$ }&{\small 0.82}  \\
  & &  {\scriptsize{($0.05$)}}  & {\scriptsize{($0.01$)}}& {\scriptsize{($0.57\times 10^{-2}$)}}&{\scriptsize{($0.04 $)}}&{\scriptsize{($0.10\times 10^{-2}  $)}} & {\scriptsize{($0.10  $)}} \\
\hline
\hline
\multirow{6}{*}{$\mathcal{C}$}&  \multirow{2}{*}{ 500} & {\small $2.50\times 10^{-2}$ }   & {\small 0.74}&{\small $1.12\times 10^{-2}$ }&{\small 0.45 } &{\small $0.86\times 10^{-2}$ }&{\small 0.56}  \\
&   &  {\scriptsize{($0.16$)}}  & {\scriptsize{($0.07$)}}& {\scriptsize{($0.27\times 10^{-2}$)}}&{\scriptsize{($ 0.01$)}}&{\scriptsize{($  0.34\times 10^{-2}$)}} & {\scriptsize{($ 0.09$)}} \\
\cline{2-8}
& \multirow{2}{*}{ 1000} & {\small $1.00\times 10^{-2}$ }   & {\small 0.68}&{\small $0.74\times 10^{-2}$ }&{\small 0.59} &{\small $0.62\times 10^{-2}$ }&{\small 0.62}  \\
  & &  {\scriptsize{($0.10$)}}  & {\scriptsize{($0.04$)}}& {\scriptsize{($0.32\times 10^{-2}$)}}&{\scriptsize{($0.03  $)}}&{\scriptsize{($ 0.24\times 10^{-2} $)}} &{\scriptsize{($0.09$)}} \\
  \cline{2-8}
& \multirow{2}{*}{ 5000} & {\small $0.52\times 10^{-2}$ }   & {\small 0.54}&{\small $0.73\times 10^{-2}$ }&{\small 0.93 } &{\small $0.29\times 10^{-2}$ }&{\small 0.74}  \\
  & &  {\scriptsize{($0.07$)}}  & {\scriptsize{($0.01$)}}& {\scriptsize{($0.52\times 10^{-2}$)}}&{\scriptsize{($0.03  $)}}&{\scriptsize{($ 0.11\times 10^{-2} $)}} &{\scriptsize{($0.09$)}} \\
\hline
\end{tabular}\caption{\label{Table3} \footnotesize{Comparaison of the different adaptive estimators for the Cauchy distribution.}}
\end{center}

\end{table}

\textbf{Comparaison of the different methods. } Tables \ref{Table1}, \ref{Table2} and \ref{Table3} show that all the procedures behave as expected; the $\lk^{2}$-risks decreases with $n$ and are smaller in the case of an ordinary smooth deconvolution problem than in the case of a super smooth deconvolution problem. The estimator with the smallest risk is the oracle, and the penalized risks are most of the time smaller than our procedure which is consistent with the fact that our procedure has a logarithmic loss and is asymptotic. More precisely  for small values of $n$ our procedure does not perform as well as the penalized method. 
But for larger values of $n$ it is competitive. We can exhibit particular cases where our procedure is more stable in the choice of the hyper parameter than the penalized procedure, even on large sample sizes (see Figure \ref{FigComp} for example).
This is due to the fact that the penalized constant $K$ that is suitable for small values of $n$ is different than for larger values of $n$. In practice a logarithmic term in $n$ is added in the penalty term, that is theoretically unnecessary and entails a logarithmic loss but improves the numerical results. If we add this logarithmic term (we replace $K=2$ with $\tilde K\log(n)^{2.5}$ with $\tilde K=0.3$ and the multiplying $\log(n)^{2.5}$ factor as suggested in Comte et al. \cite{comte2007finite}). This second penalty procedure performs well for all values of $n$ and when $n$ gets large it has similar  performances as our procedure (see Table \ref{Table4}). For our procedure, changing $\kappa$ for smaller values of $n$ does not improve the results.

\begin{figure}\begin{center}
\begin{tabular}{ccc}
\multicolumn{3}{c}{\textbf{Ordinary smooth case $\eps\sim\mathcal{G}(2,1)$}}\\
\multicolumn{3}{c}{{Penalized adaptive estimator}}\\
$K=2$ & $K=5$ & $K=8$\\
\includegraphics[scale=0.15]{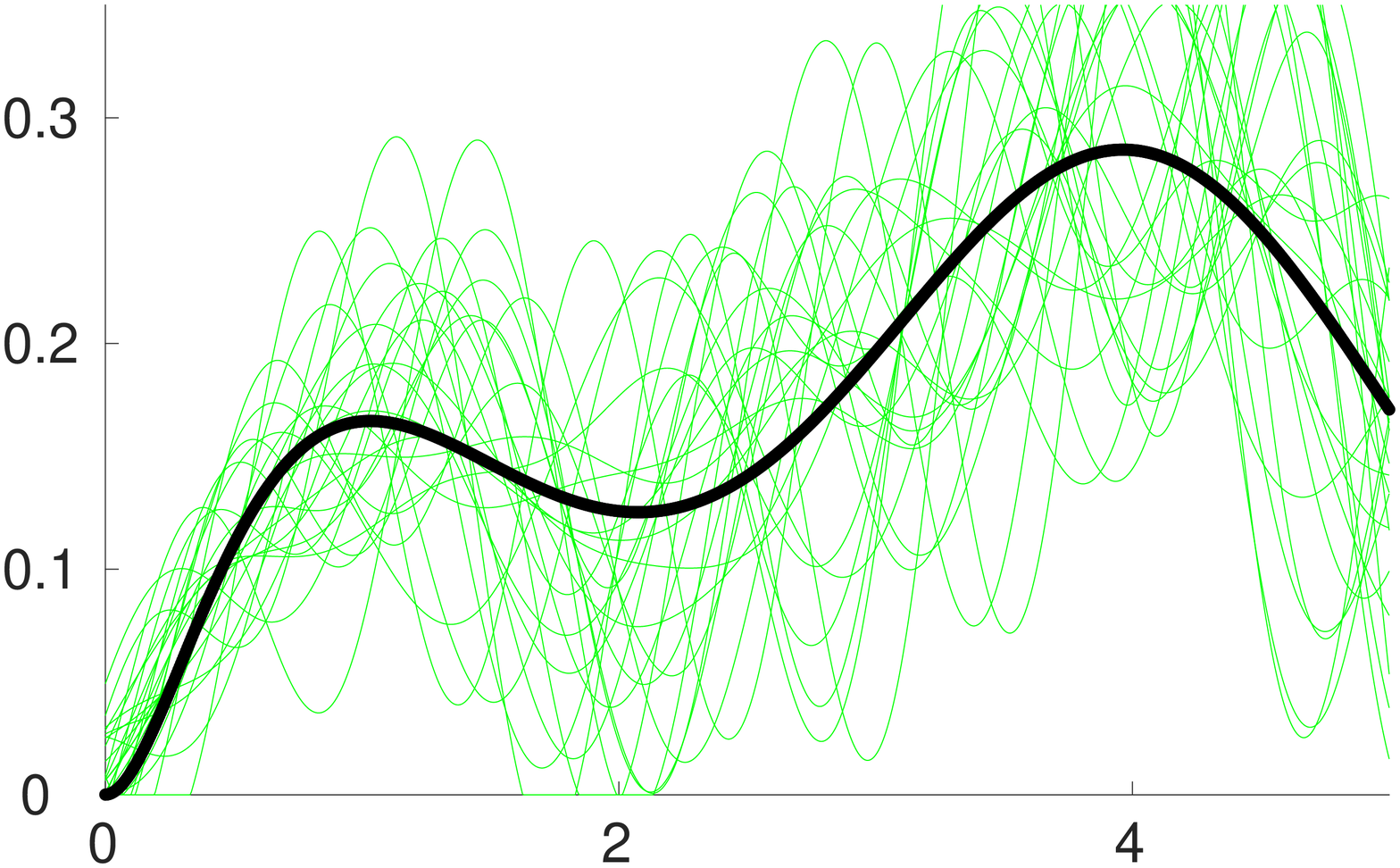}  &\includegraphics[scale=0.15]{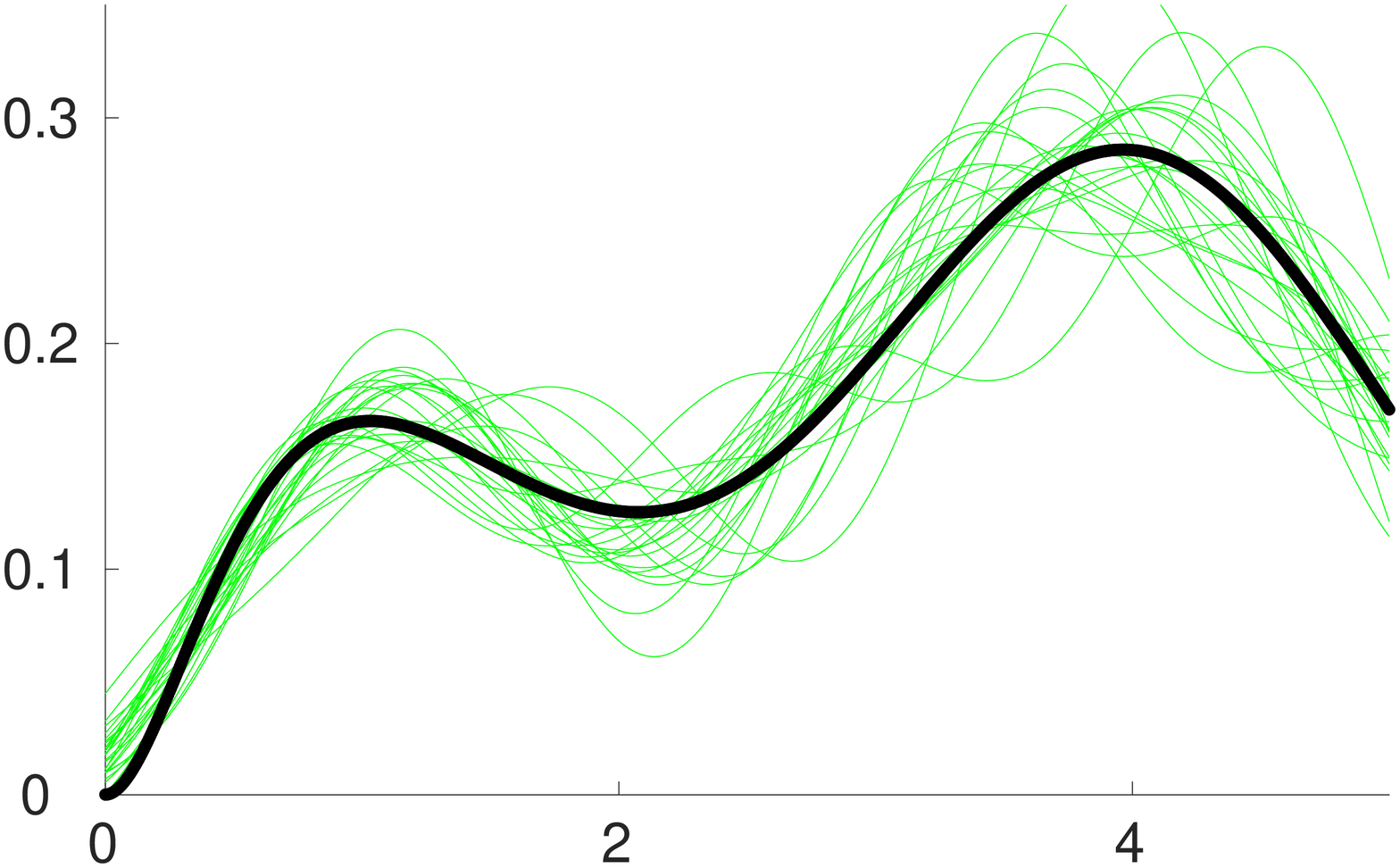}  &\includegraphics[scale=0.15]{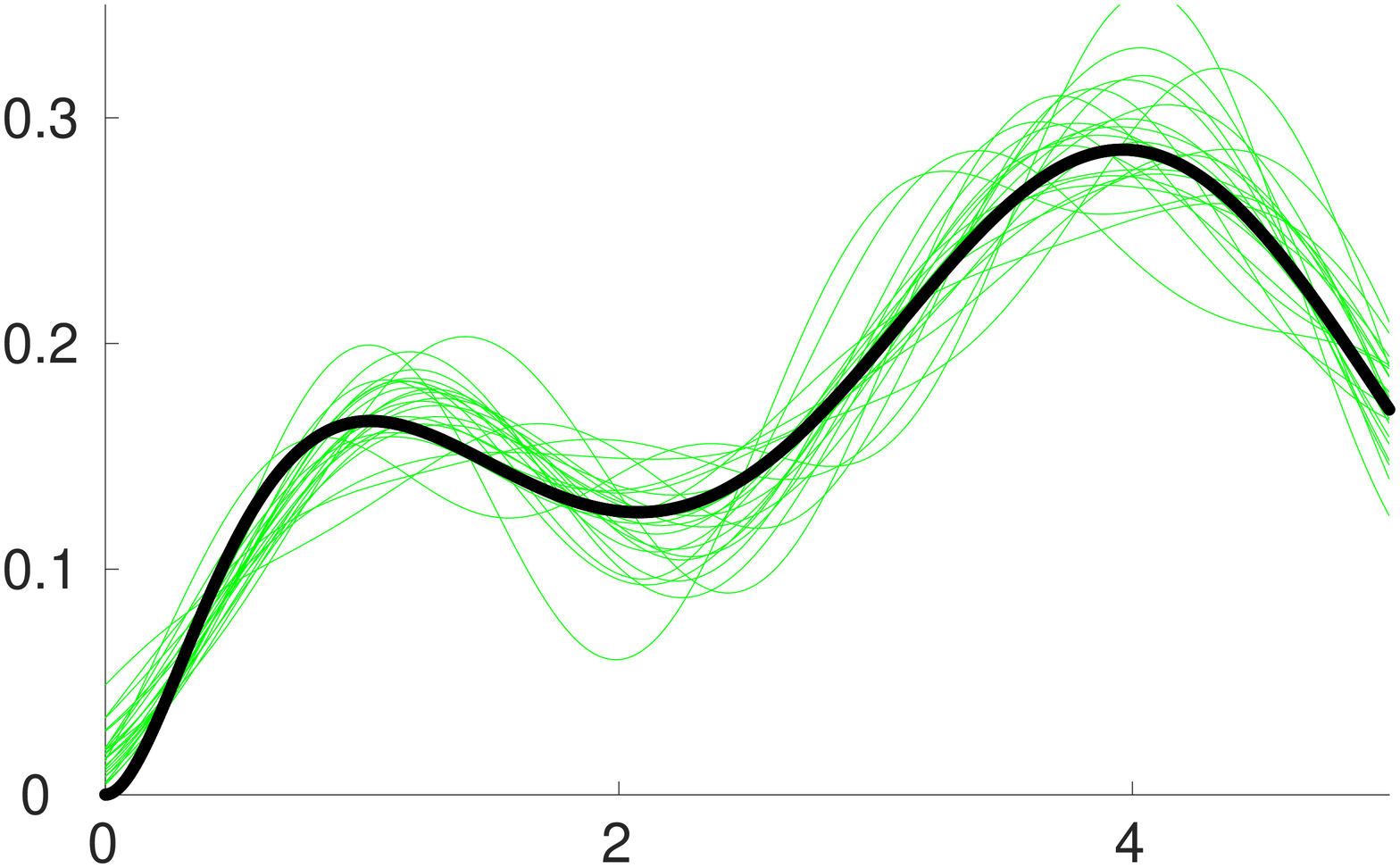}  \\
\multicolumn{3}{c}{{Our adaptive estimator}}\\
$\kappa=2$ & $\kappa=5$ & $\kappa=8$\\

\includegraphics[scale=0.15]{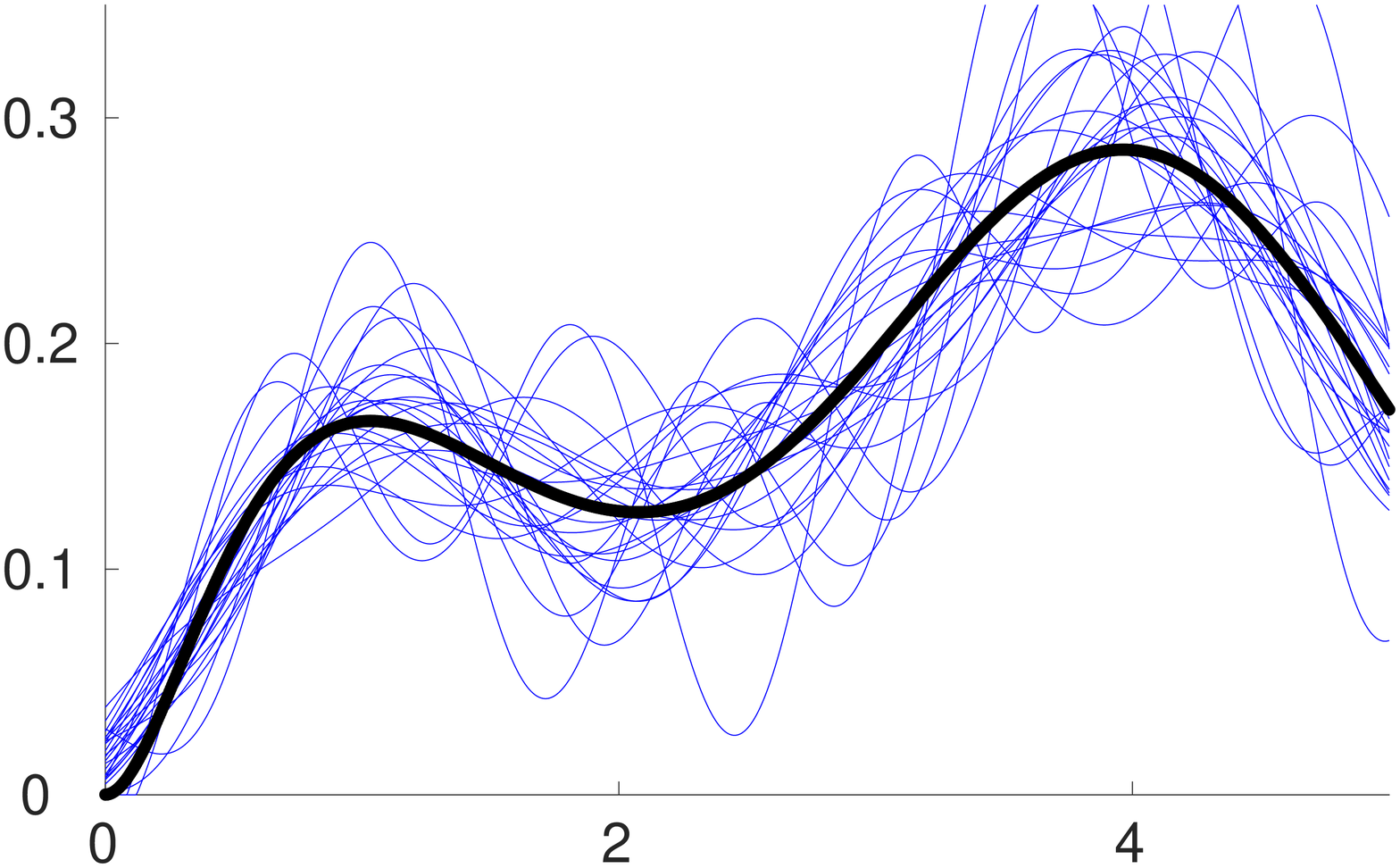}  &\includegraphics[scale=0.15]{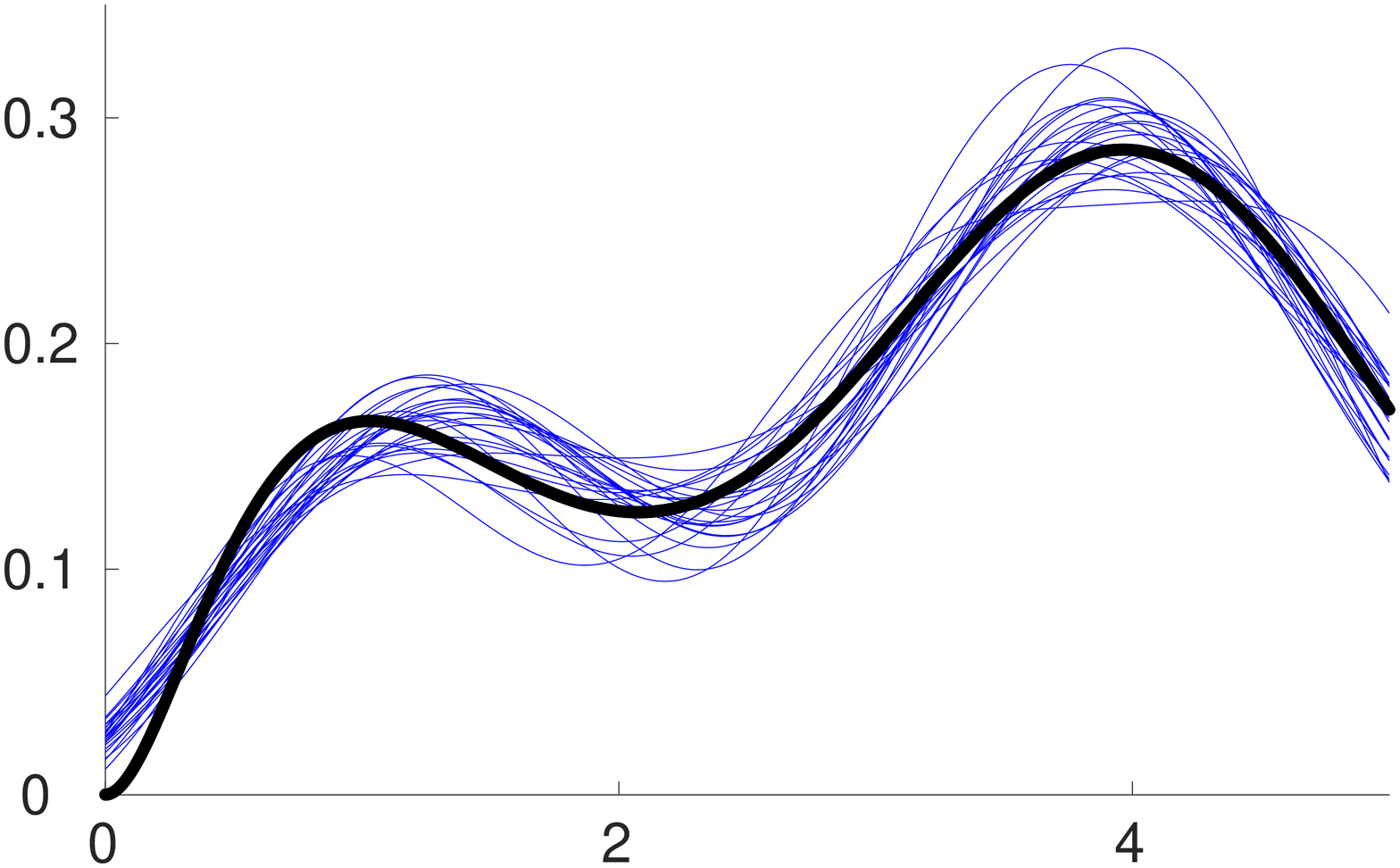}  &\includegraphics[scale=0.15]{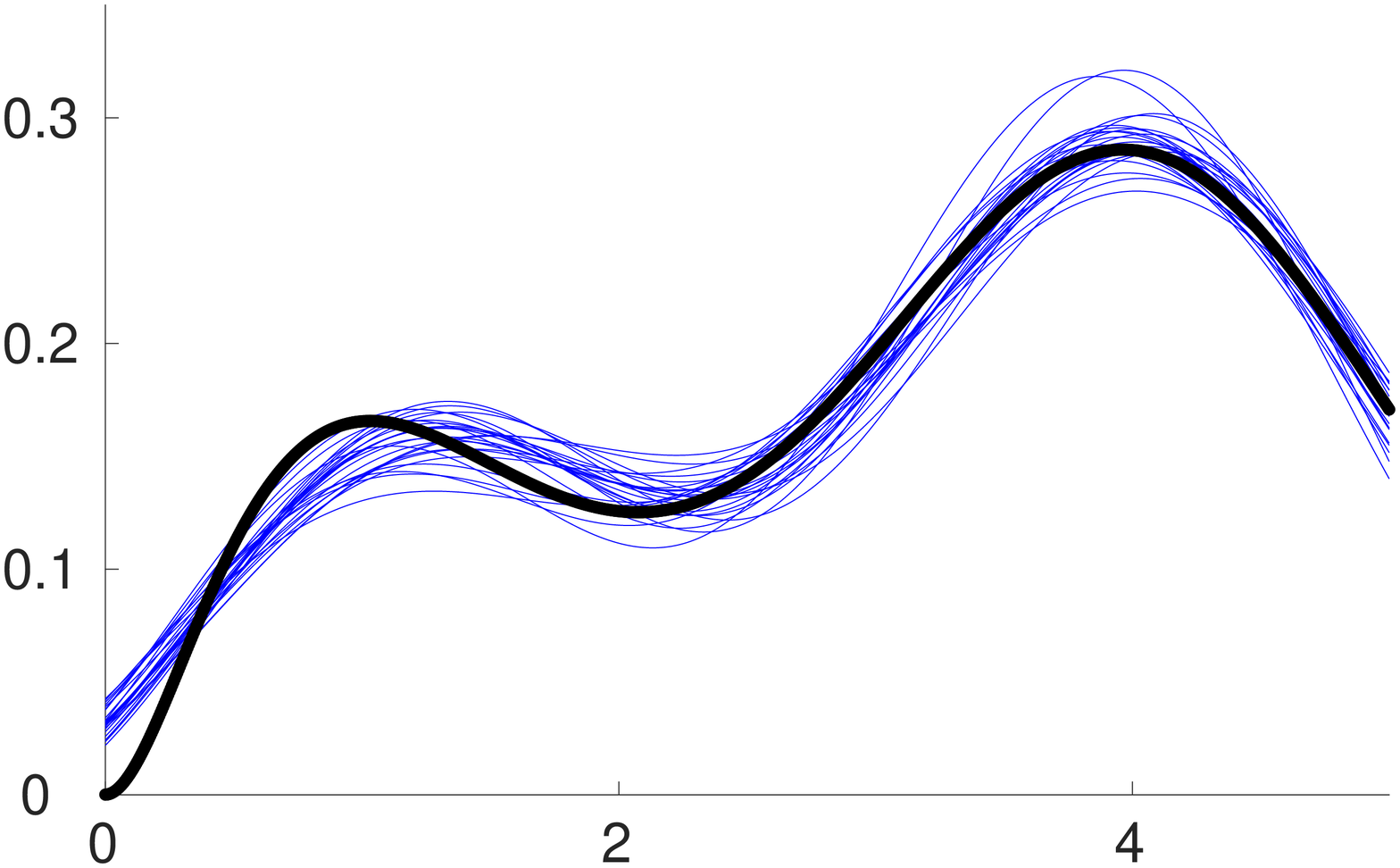}  \\
\multicolumn{3}{c}{\textbf{Super smooth  case $\eps\sim\mathcal{C}(0,1)$}}\\
\multicolumn{3}{c}{{Penalized adaptive estimator}}\\
$K=2$ & $K=5$ & $K=8$\\
\includegraphics[scale=0.15]{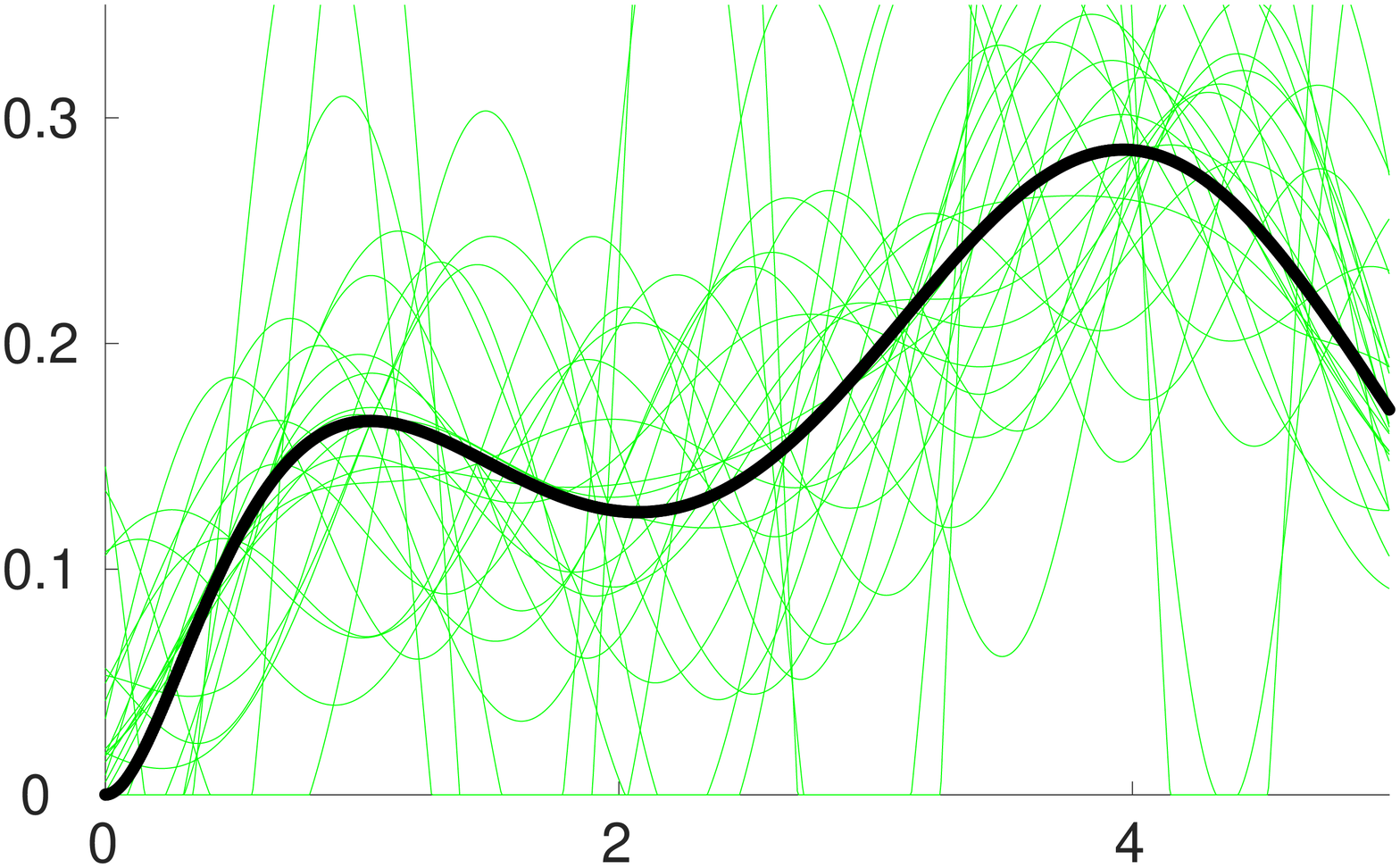}  &\includegraphics[scale=0.15]{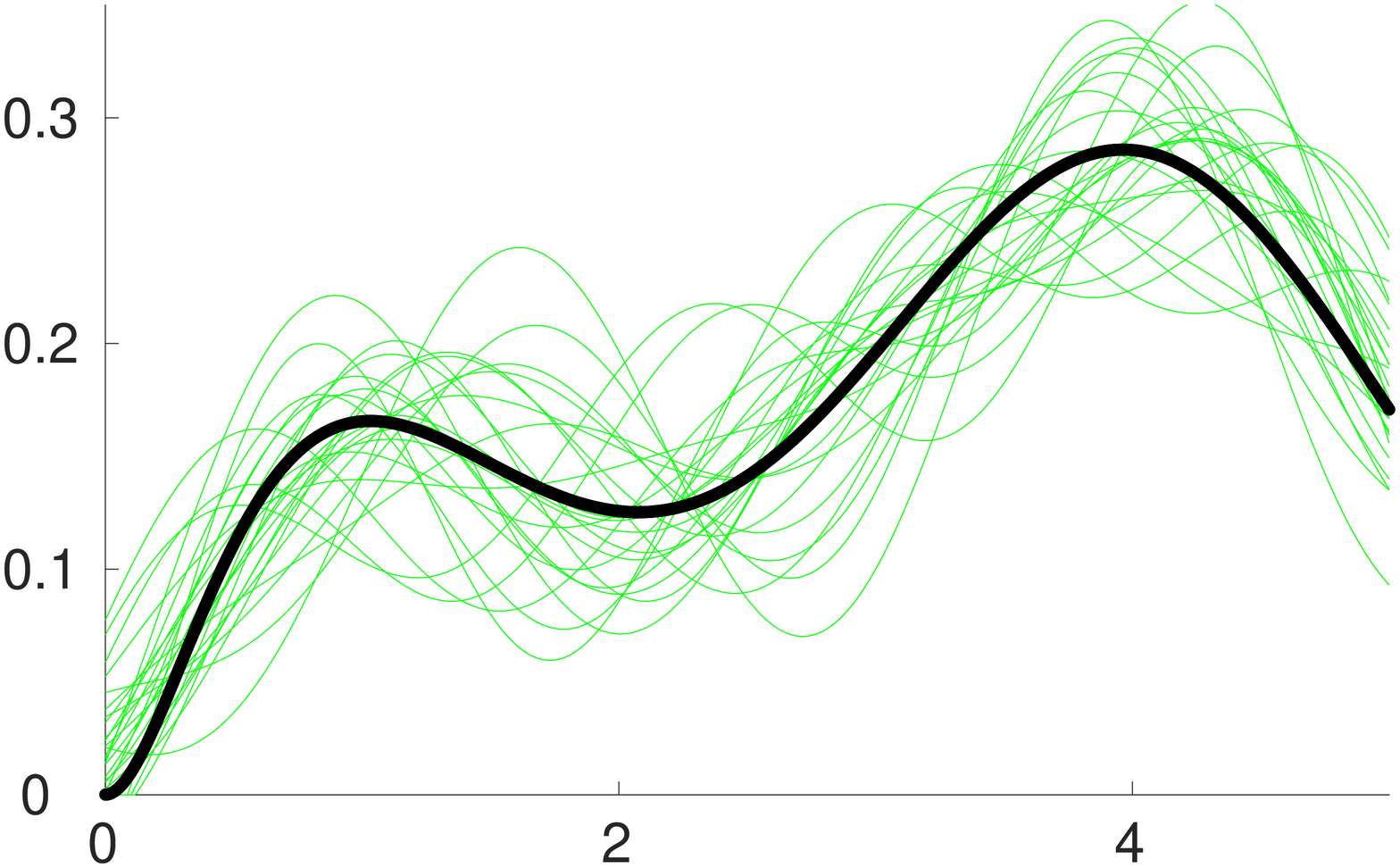}  &\includegraphics[scale=0.15]{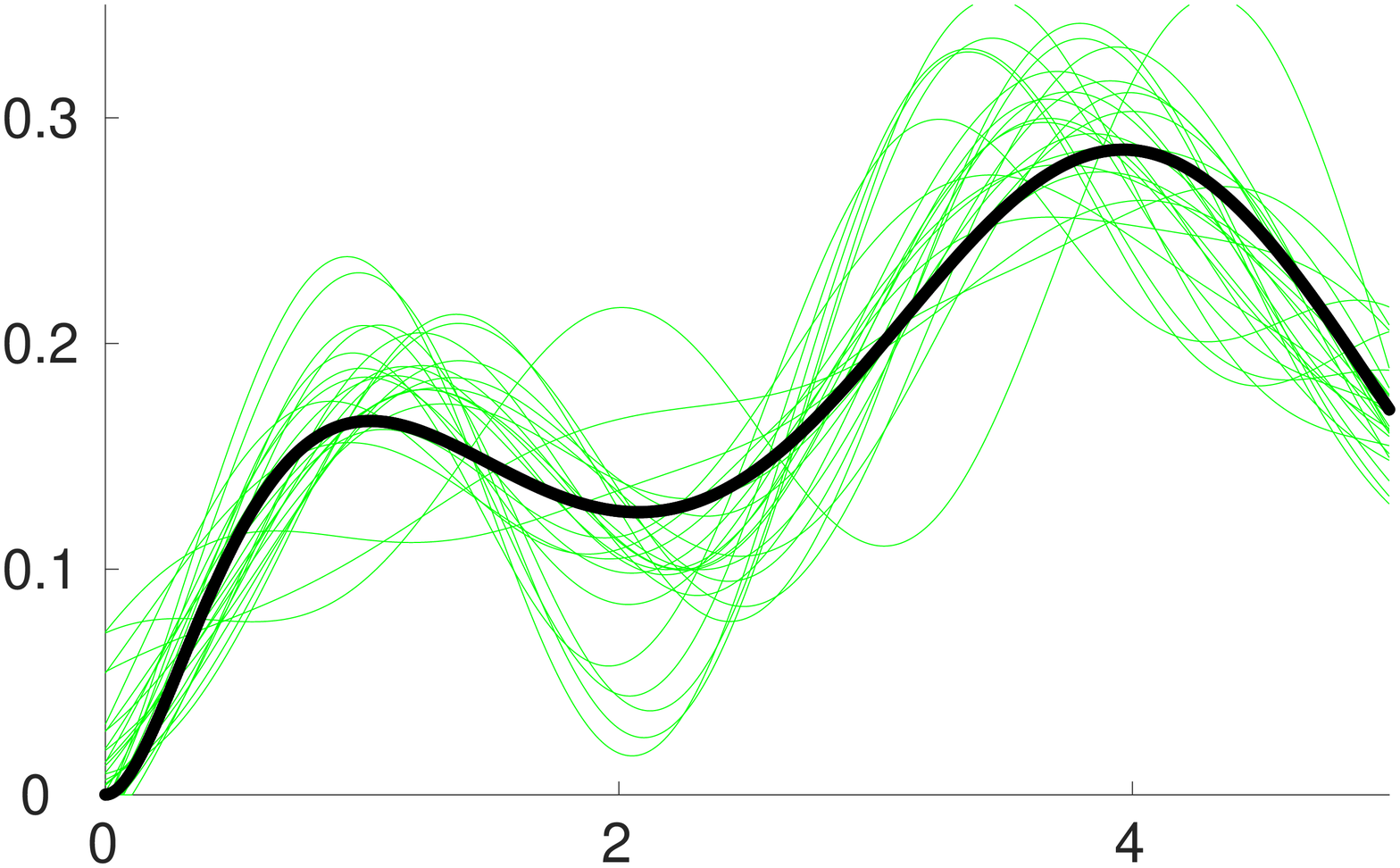}  \\
\multicolumn{3}{c}{{Our adaptive estimator}}\\
$\kappa=2$ & $\kappa=5$ & $\kappa=8$\\

\includegraphics[scale=0.15]{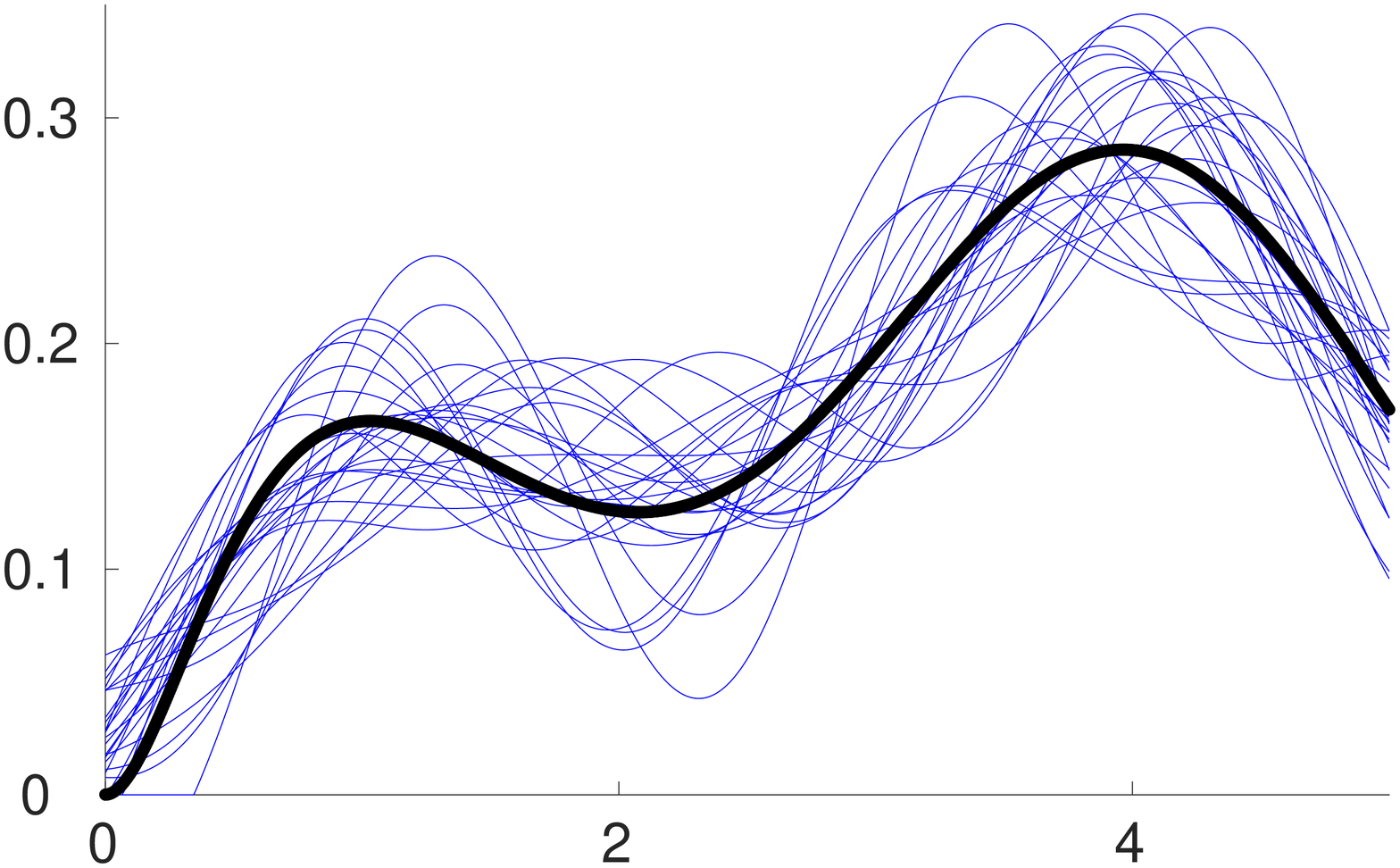}  &\includegraphics[scale=0.15]{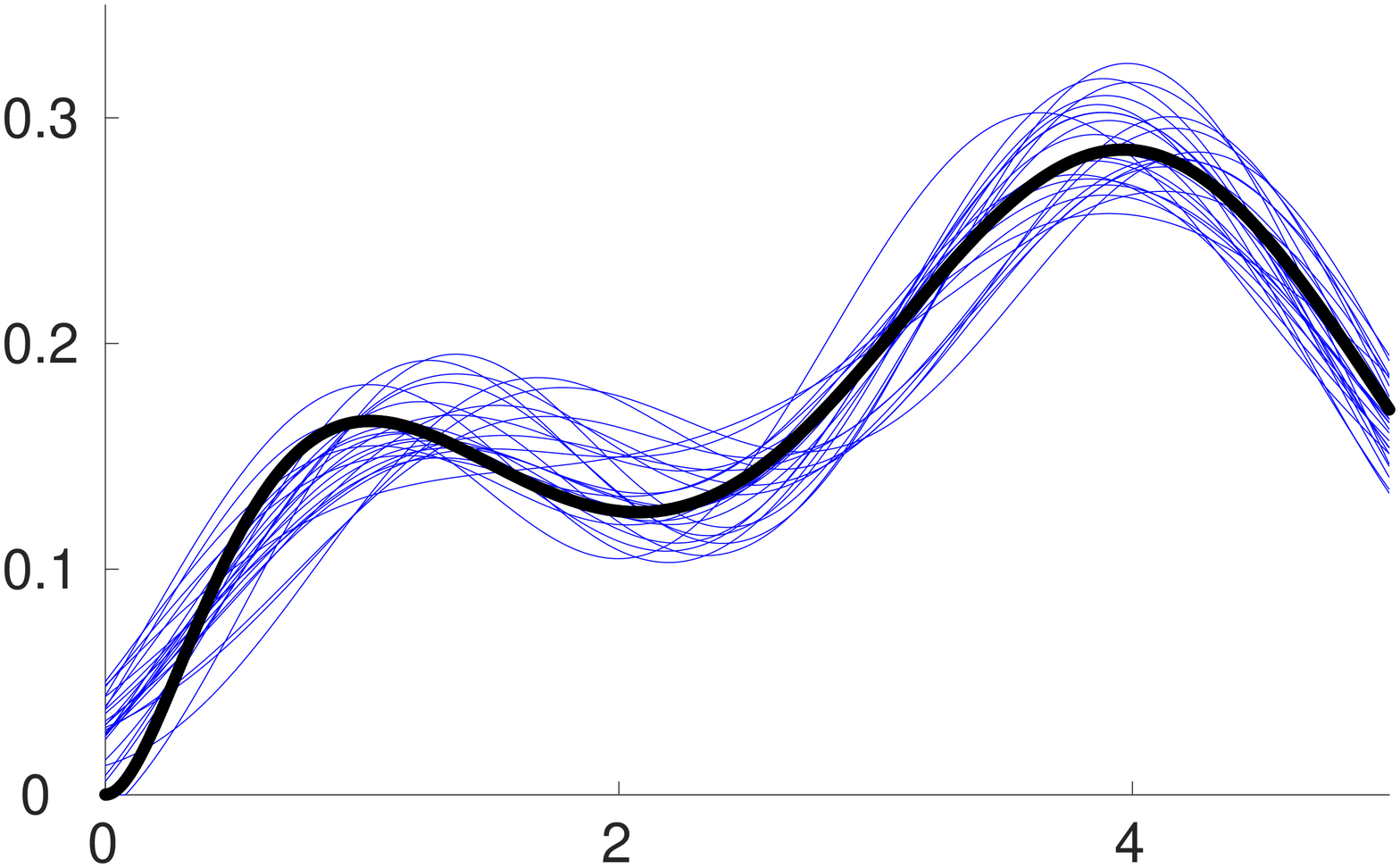}  &\includegraphics[scale=0.15]{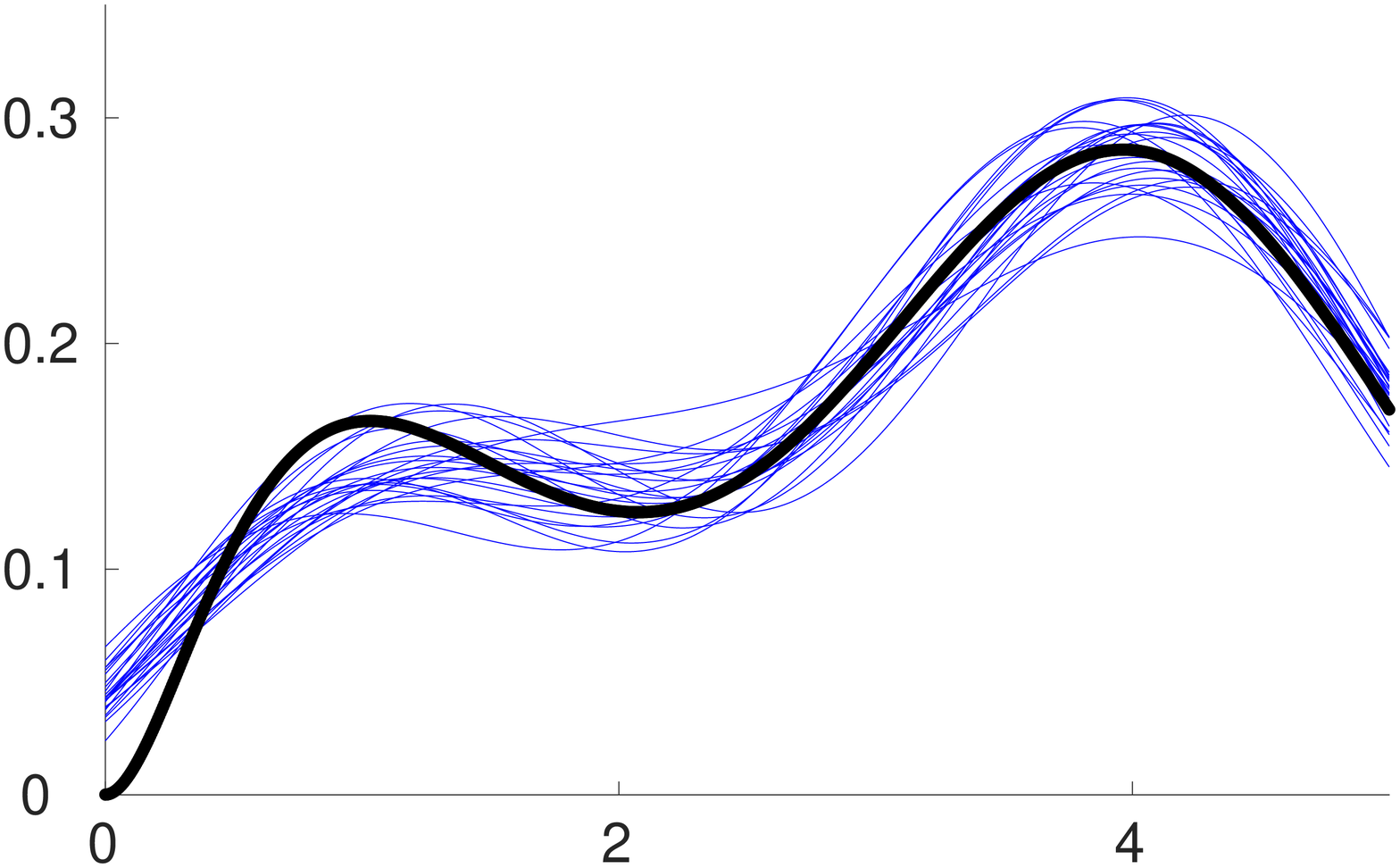}  \\\end{tabular}
\caption{\label{FigComp}Comparison for different values of $K$ and $\kappa$ the penalized estimator (green) and our adaptive estimator (blue). Estimation of $f\sim\big(0.3\mathcal{G}(3,\frac{1}{2})+0.7\mathcal{G}(4,1)\big)$ (bold black) from $n=10000$ observations. }

\end{center}
\end{figure}

\begin{table}
\begin{center}
\begin{tabular}{c c cccccc}
\multirow{2}{*}{\diagbox{$f_{\eps}$}{$f$} }& & \multicolumn{2}{c}{$\mathcal{G}$} & \multicolumn{2}{c}{$\mathcal{M}$}& \multicolumn{2}{c}{$\mathcal{C}$} \\
 & $n$ & $R_{\tilde{pen}}$ & $\tilde m_{\tilde{pen}}$ &$R_{\tilde{pen}}$ & $\tilde m_{\tilde{pen}}$ &$R_{\tilde{pen}}$ & $\tilde m_{\tilde{pen}}$ \\
 \hline
 \hline
\multirow{6}{*}{$\Gamma(2,1)$}&  \multirow{2}{*}{ 500} & {\small $1.17\times 10^{-2}$ }   & {\small 0.72}&{\small $0.63\times 10^{-2}$ }&{\small 0.69} &{\small $0.83\times 10^{-2}$ }&{\small 0.62 }  \\
&   &  {\scriptsize{($0.81\times 10^{-2}$)}}  & {\scriptsize{($0.03$)}}& {\scriptsize{($0.51\times 10^{-2}$)}}&{\scriptsize{($0.02$)}}&{\scriptsize{($  0.46\times 10^{-2}$)}} &{\scriptsize{($0.005  $)}} \\
\cline{2-8}
& \multirow{2}{*}{ 1000} & {\small $0.94\times 10^{-2}$ }   & {\small 0.82}&{\small $0.40\times 10^{-2}$ }&{\small 0.75} &{\small $0.59\times 10^{-2}$ }&{\small 0.69}  \\
  & &  {\scriptsize{($0.63\times 10^{-2}$)}}  & {\scriptsize{($0.04$)}}& {\scriptsize{($0.30\times 10^{-2}$)}}&{\scriptsize{($0.02 $)}}&{\scriptsize{($0.32\times 10^{-2}  $)}} & {\scriptsize{($0.03  $)}} \\
\cline{2-8}
& \multirow{2}{*}{ 5000} & {\small $0.62\times 10^{-2}$ }   & {\small 1.08}&{\small $0.20\times 10^{-2}$ }&{\small 0.94} &{\small $0.31\times 10^{-2}$ }&{\small 0.91}  \\
  & &  {\scriptsize{($0.62\times 10^{-2}$)}}  & {\scriptsize{($0.01$)}}& {\scriptsize{($0.16\times 10^{-2}$)}}&{\scriptsize{($0.02$)}}&{\scriptsize{($0.19\times 10^{-2}  $)}} & {\scriptsize{($0.02  $)}} \\
\hline
\hline
\multirow{6}{*}{$\mathcal{C}$}&  \multirow{2}{*}{ 500} & {\small $1.34\times 10^{-2}$ }   & {\small 0.45}&{\small $0.60\times 10^{-2}$ }&{\small 0.45 } &{\small $1.31\times 10^{-2}$ }&{\small 0.41}  \\
&   &  {\scriptsize{($0.43\times 10^{-2}$)}}  & {\scriptsize{($0.01$)}}& {\scriptsize{($0.25\times 10^{-2}$)}}&{\scriptsize{($ 0.01$)}}&{\scriptsize{($  0.24\times 10^{-2}$)}} & {\scriptsize{($ 0.02$)}} \\
\cline{2-8}
& \multirow{2}{*}{ 1000} & {\small $0.91\times 10^{-2}$ }   & {\small 0.59}&{\small $0.74\times 10^{-2}$ }&{\small 0.51} &{\small $0.94\times 10^{-2}$ }&{\small 0.45}  \\
  & &  {\scriptsize{($0.43\times 10^{-2}$)}}  & {\scriptsize{($0.02$)}}& {\scriptsize{($0.51\times 10^{-2}$)}}&{\scriptsize{($0.003  $)}}&{\scriptsize{($ 0.14\times 10^{-2} $)}} &{\scriptsize{($0.002$)}} \\
  \cline{2-8}
& \multirow{2}{*}{ 5000} & {\small $0.56\times 10^{-2}$ }   & {\small 0.81}&{\small $0.21\times 10^{-2}$ }&{\small 0.75 } &{\small $0.35\times 10^{-2}$ }&{\small 0.65}  \\
  & &  {\scriptsize{($0.33\times 10^{-2}$)}}  & {\scriptsize{($0.05$)}}& {\scriptsize{($0.15\times 10^{-2}$)}}&{\scriptsize{($0.001  $)}}&{\scriptsize{($ 0.11\times 10^{-2} $)}} &{\scriptsize{($0.02$)}} \\
\hline
\end{tabular}\caption{\label{Table4} \footnotesize{Risks and selected cutoff of the penalized procedure with an additional logarithmic term in the penalty. Notation $\mathcal{M}$ stands for the mixture $0.7\mathcal{N}(4,1)+0.3\Gamma(4,\frac12).$ }}
\end{center}

\end{table}

\section{Decompounding\label{sec:exP}}

\subsection{Statistical setting }

Let $Z$ be a compound Poisson process with intensity $\lambda>0$ and jump density $f$, i.e.
$$Z_{t}:=\sum_{j=1}^{N_{t}}X_{j},\quad t\geq 0$$ where $N$ is an homogeneous Poisson process with intensity $\lambda$ and independent of the i.i.d. variables $(X_{j})$ with common density $f$. One trajectory of $Z$ is observed at sampling rate $\Delta$ over $[0,T]$, $T=n\Delta$, $n\in \N$. non-parametric estimation of $f$, or more generally of the Lévy density $\lambda f$ has been the subject of many papers, among others, \cite{MR2001642,coca2017efficient,comte2014nonparametric,Duval,van2007kernel} for decompounding, \cite{gugushvili2014nonparametric} for the multidimensional setting, and \cite{bec2015adaptive,comte2010nonparametric,gugushvili2012nonparametric,Kappus14,NR09}; a review is also available in the textbook \cite{BCGM} for the non-parametric estimation of the Lévy density. 

We observe $Z$ at the time points $j\Delta, \ j=1, \ldots, n$, for $\Delta>0$, denote the $j-$th increment by $Y_{j\Delta}=Z_{j\Delta}-Z_{(j-1)\Delta}$. We aim at estimating $f$ from the increments $(Y_{j\Delta},j=1,\ldots,n)$. Consider $\phi$ the characteristic function of $X_{1}$ and $\phi_\Delta$ the characteristic function of $Z_{\Delta}=Y_{\Delta}$. The Lévy-Kintchine formula relates them as follows
\[
\phi_\Delta(u) =\exp\big(\Delta  \lambda( \phi(u) - 1)    \big),\quad u\in\R. 
\] As $Z$ is a compound Poisson process, $|\phi_{\Delta}|$ is  bounded from below by $e^{-2  \Delta}$, which remains bounded away from 0. Moreover, if $\E[|X_{1}|]<\infty$ it holds that $\phi$ is differentiable and we can then define the distinguished logarithm of $\phi_{\Delta}$ (see Lemma 1 in \cite{duval2017nonparametric})  \begin{align}
\label{eq:Ff}\phi(u)=1+\frac{1}{ \lambda\Delta}\Log\big(\phi_{\Delta}(u)\big),\quad\mbox{where }\Log(\phi_{\Delta}(u))=\int_{0}^{u}\frac{\phi_{\Delta}^{\prime}(z)}{\phi_{\Delta}(z)}\d z,\quad u\in\R.
\end{align} 
For simplicity, we assume that the intensity $\lambda$ is known: $\lambda=1$. 
Following \eqref{eq:Ff}, an estimator of $\phi$ is hence given by 
\begin{align}\label{eq fhatstar1}
\hat{\phi}_{n}(u)  =1+ \frac{1 }{ \Delta}\Log (\hat{\phi}_{\Delta,n}(u) ),\quad u\in\R
\end{align}
with 
\[
\Log( \hat{\phi}_{\Delta,n}(u) ): = \int_0^u  \frac{\hat{\phi}_{\Delta,n}'(z)}{ \hat{\phi}_{\Delta,n}(z)         }  \d z,\quad \hat{\phi}_{\Delta,n}(z)=\frac{1}{n}\sum_{j=1}^{n}e^{izY_{j\Delta}}\ \mbox{and}\ \hat{\phi}_{\Delta,n}'(z)=\frac{1}{n}\sum_{j=1}^{n}iY_{j\Delta}e^{izY_{j\Delta}}.
\]
The quantity $\Log(\hat\phi_{\Delta,n})$ appearing in $\hat\phi_{n}$ might be unbounded: if $\phi_{\Delta}$ never cancels, it may not be the case of its estimator $\hat\phi_{\Delta,n}$. Usually, to prevent this issue a local threshold is used and $\frac{1}{\hat\phi_{\Delta,n}(v)}$ is replaced with $\frac{1}{\hat\phi_{\Delta,n}(v)}\mathds{1}_{|\hat\phi_{\Delta,n}(v)|>r_{n}}$, for some vanishing sequence $r_{n}$ (see {e.g.} Neumann and Rei\ss~\cite{NR09}). Here we do not use a local threshold inside the integral, but a global threshold so that $\widehat{\Log(\phi_{\Delta,n})}=\Log(\hat\phi_{\Delta,n})$. Define \begin{align}\label{eq fhatstar2}
\widetilde\phi_{n}(u):=\hat\phi_{n}(u)\mathds{1}_{|\hat\phi_{n}(u)|\leq 4},\quad u\in\R,
\end{align} where $\hat\phi_{n}$ is given by \eqref{eq fhatstar1}. The choice of a threshold equal to 4 is technical (see the proof of Theorem \ref{thm Poisson}).  Cutting off in the spectral domain an applying a Fourier inversion provides the estimator if $f$
\begin{align}\label{eq:fhatP}
\hat{f}_{m,\Delta} (x) =\frac{1}{2\pi}  \int_{-m}^{m} e^{-iux}  \tilde{\phi}_{n}(u) \d u,\quad x\in\R. \end{align}

\subsection{Adaptive upper bound}

\subsubsection{Upper bound and discussion on the rate.}

\begin{theorem}  \label{thm Poisson} 
Assume that $\E[X_{1}^{2}]<\infty$, $\Delta\leq\frac14\log(n\Delta)$ and $n\Delta\to\infty$ as $n\to\infty.$ Then, for any  $m\geq 0$  it holds
\begin{align*}\E\big[\| \widehat{f}_{m,\Delta}-f\|^2\big]&\leq\|f_m-f\|^2+\frac{2}{{n\Delta}}\int_{-m}^{m}\frac{\d u}{|\phi_{\Delta}(u)|^{2}}+2\ 5^{2}\E[X_{1}^{2}]\frac{m}{n\Delta}+2^{3}5^{2} \frac{m^{2}}{(n\Delta)^{2}}. \end{align*}\end{theorem}

The contraint $\Delta\leq \frac14\log (n\Delta)$ is fulfilled for any bounded $\Delta$ as $n\Delta\to\infty$. Moreover it allows $\Delta$ to be such that $\Delta:=\Delta_{n}\to 0$ and $\Delta_{n}\to \infty$, not too fast. This last point is interesting. To the knowledge of the author, an estimator that is optimal simultaneously when $\Delta$ is fixed or  vanishing and consistant, presumably optimal up to a logarithmic loss, when $\Delta$ tends to infinity, has not been investigated. Moreover, there are no results on the estimation of the jump density of a compound Poisson process when the sampling rate goes to infinity. In the remaining of this paragraph, we discuss the different rates of convergence implied by Theorem \ref{thm Poisson} according to the behavior of $\Delta$.

\paragraph{Discussion on the rates. } 
The upper bound derived in Theorem \ref{thm Poisson} is the sum of four terms: a bias, plus two variance terms $V\asymp\frac{e^{4\Delta}m}{n\Delta}$ (using that $|\phi_{\Delta}(u)|\geq e^{-2\Delta}$) and $V^{\prime}\asymp\frac{m}{n\Delta}$, which is always smaller or of the same order as $V$, and a remainder. 
Assume that $f$ lies is the Sobolev ball $\mathscr{S}(\beta,L)$ (see \eqref{eq:Sob}). Then, the bias $\|f-f_{m}\|^{2}$ has asymptotic order $m^{-2\beta}$ and we may derive the following rates of convergence.
\begin{itemize}
\item \textbf{Microscopic and mesoscopic regimes.} Let $\Delta=\Delta_{n}$ be such that $\Delta_{n}\rightarrow \Delta_{0}\in[0,\infty)$ such that $n\Delta_{n}\to \infty$. Then, the bias variance compromise leads to the choice $m^{\star}=(e^{-4\Delta_{0}}n\Delta_{0})^{\frac{1}{2\beta+1}}$ and to the rate of convergence $\big(e^{-4\Delta_{0}}n\Delta_{0})^{-\frac{2\beta}{2\beta+1}}$ that matches the optimal rates of convergence as $\Delta_{0}$ is fixed or tending to 0. Indeed, the rate is in $T^{-\frac{2\beta}{2\beta+1}}$ , with $T=n\Delta_{n}$ denoting the time horizon, it is clearly rate optimal as it corresponds to the optimal rate of convergence to estimate the jump density of a compound Poisson process from continuous observations ($\Delta=0$). The constant $e^{-4\Delta_{0}}$ appearing in the rate depends exponentially on $\Delta_{0}$, which asymptotically as little effect but in practice deteriorates the numerical performances. \item\textbf{Macroscopic regime.}
Let $\Delta=\Delta_{n}\rightarrow \infty$ such that $\Delta_{n}\leq\frac14\log (n\Delta_{n})$ The variance term $V$ tends to 0, so that the estimator is consistent. Heuristically, if $\Delta$ goes to infinity the central limit theorem states that $Y_{\Delta}$ is close in law to a parametric Gaussian variable, e.g. if $f$ is centered and with unit variance it holds that:
$\sqrt{\Delta}^{-1}{Y_{\Delta}}\xrightarrow[\Delta\to\infty]{d}\mathcal{N}(0,1).$ Consequently, the fact that $f$ can be constantly estimated is non trivial. Duval \cite{duval2014no} establishes that if $\Delta=O((n\Delta)^{\delta})$, for some $\delta\in(0,1)$, i.e. when $\Delta_{n}$ goes rapidly to infinity, there exists no consistent non-parametric estimator of $f$. The fact that estimation is impossible when $\Delta$ goes too rapidly to infinity was established through an asymptotic equivalence result. In this case it is always possible to build two different compound Poisson processes for which the statistical experiments generated by their increments are asymptotically equivalent.  Therefore, the result of Theorem \ref{thm Poisson} is new in that context.  We may distinguish two additional regimes:

\begin{enumerate}
\item
{Slow macroscopic regime.  } If ${\Delta_{n}=o\big(\log(n\Delta_{n})\big) },$ the choice $m^{\star}=\big(e^{-4\Delta_{n}}n\Delta_{n})^{\frac{1}{2\beta+1}}$ leads to the rate of convergence $\big(e^{-4\Delta_{n}}n\Delta_{n}\big)^{-\frac{2\beta}{2\beta+1}}.$ There is no lower bound in the literature to ensure if this rate is optimal. However if $\Delta$ goes slowly to infinity, for example if $\Delta_{n}=\log(\log(n\Delta_{n}))$, then the rate is $\big({(\log(n\Delta_{n}))^{-4}n\Delta_{n}}\big)^{-\frac{2\beta}{2\beta+1}},$ which is rate optimal, up to the logarithmic loss that may not be optimal.

\item
{Intermediate macroscopic regime.} Let $\Delta_{n}=\delta \log(n\Delta_{n})$, $0<\delta<1/4$, then $m^{\star}=(n\Delta_{n})^{\frac{1-4\delta}{2\beta+1}},$ leading to the rate $(n\Delta_{n})^{-\frac{2\beta(1-4\delta)}{2\beta+1}}.$ This rate deteriorates as $\delta$ increases. The limit $\delta=1/4$ imposed by Theorem \ref{thm Poisson} may not be optimal, no lower bound adapted to this case exists in the literature.\end{enumerate} 
The interest of the macroscopic regime is mainly theoretical as in practice if $\Delta$ is a large constant to get $e^{-4\Delta}n\Delta$ large one should consider a huge amount $n$ of observations. However, this regime enlightens the role of the sampling rate $\Delta$ in the non-parametric estimation of the jump density.
\end{itemize}

\subsubsection{Adaptive choice of the cutoff parameter}

We consider the optimal cutoff $\overline m_{n}$ given by
\[\overline m_{n}\in\underset{m\geq0}{\mbox{arginf}}\Big\{\|f_m-f\|^2+\frac{2}{{n\Delta}}\int_{-m}^{m}\frac{\d u}{|\phi_{\Delta}(u)|^{2}}+2\ 5^{2}\E[X_{1}^{2}]\frac{m}{n\Delta}+2^{3}5^{2} \frac{m^{2}}{(n\Delta)^{2}}\Big\}. \]
Following the previous strategy, the upper bound given by Theorem \ref{thm Poisson} is optimal, at least for $\Delta\to[0,\infty)$. The leading variance terms is in $\frac{me^{4\Delta}}{n\Delta}$, we differentiate in $m$ the upper bound to find that the optimal cutoff $m^{\star}\asymp\overline m_{n}$ is such that:
\[
|\phi(\overline m_{n})|^2 = \frac{e^{4\Delta} }{\Delta n}    \Leftrightarrow |\phi(\overline m_{n})|^2 = \frac{e^{4\Delta}}{ \Delta n},
\] which has an empirical version, we select $\hat{m}_{n}$ accordingly. As in the deconvolution setting, we modify the estimator $\tilde \phi_{n}$ in \eqref{eq fhatstar2} which is set to 0 when the estimator of $|\phi |$ is smaller that $1/\sqrt{n\Delta},$ meaning that the noise is dominant. 
Define \[
\overline\phi_{n}(u):=\tilde\phi_{n}(u)\mathds{1}_{|\tilde\phi_{n}(u)|\geq \kappa_{n,\Delta}/\sqrt{n\Delta}},\quad u\in\R
\] where $\kappa_{n,\Delta}:=(e^{2\Delta}+\kappa\sqrt{\log (n\Delta)})$, $\kappa>0$, and the new the estimator if $f$
\[
\overline{f}_{m,\Delta} (x) =\frac{1}{2\pi}  \int_{-m}^{m} e^{-iux}  \overline{\phi}_{n}(u) \d u,\quad x\in\R.
\]
Again, Lemma \ref{lem:tildehat} ensures that Theorem \ref{thm Poisson} holds for the estimator $\overline f_{m,\Delta}.$
Finally, we introduce the empirical threshold, for some $\alpha\in(0,1]$ and $\kappa>0$
\[\hat{m}_{n} = \max\Big\{m\geq0: |\overline{\phi}_{n}(m)|= \frac{\kappa_{n,\Delta}}{\sqrt{n\Delta}}  \Big\}\wedge (n\Delta)^\alpha.
\]

\begin{theorem} \label{thm:AP}Assume that $\E[X_{1}^{4}]<\infty$, $\Delta\leq\frac14\log(n\Delta)$ and $n\Delta\to\infty$ as $n\to\infty.$ Then, for a positive constant $C_{1}$, depending on $\kappa$, $\E[X_{1}^{2}]$ and $\E[X_{1}^{4}]$, and $C_2$ a constant depending on  $\E[X_{1}^{2}]$ and $\E[X_{1}^{4}]$, it holds
\begin{align*}
\E[  \| \overline{f}_{\hat m_{n},\Delta}- f\|^2  ]  \leq   C_{1}\underset{m\in[0,(n\Delta)^{\alpha}]}{\inf}&\Big\{\|{f}_{{m} }- f\|^2+  \frac{\log(n\Delta)m}{n\Delta} +\frac{1}{{n\Delta}}\int_{-m}^{m}\limits\frac{\d u}{|\phi_{\Delta}(u)|^{2}}+ \frac{{m}^{2}}{(n\Delta)^{2}}\Big\}\\&+C_{2}\Big(\frac{1}{n\Delta}+(n\Delta)^{\alpha-{\kappa^{2}\Delta^{2}}e^{-4\Delta}}\Big). 
\end{align*} \end{theorem}
 If $\kappa>\frac{\sqrt{2}e^{2\Delta}}{\Delta}$ the last additional term is negligible, regardless the value $\alpha\leq 1$ and Theorem \ref{thm:AP} ensures that the adaptive estimator $\overline{f}_{\hat m_{n},\Delta}$ satisfies the same upper bound as in Theorem \ref{thm Poisson}.  Therefore, it is adaptive and rate optimal, up to a logarithmic term and the multiplicative constant $C_{1}$, in the microscopic and mesoscopic regimes defined above. In the macroscopic regimes such that $\Delta:=\Delta_{n}\to\infty$ such that $\Delta_{n}<\frac14\log(n\Delta_{n})$ as $n\to\infty$ the estimator is consistent. Note that to establish the adaptive upper bound we imposed a stronger assumption that $\E[X_{1}^{4}]<\infty$. 
 In the following numerical study, we recover that the procedure is stable in the choice of $\kappa$.

\subsection{Numerical results}

As for the deconvolution problem, we illustrate the performance of this adaptive estimator  for different densities $f$. We consider the same densities as for the deconvolution problem, the Cauchy density excepted as it is not covered by our procedure: it has infinite moments. We compute the adaptive $\lk^2$-risks of our procedure over 1000 Monte Carlo iterations for various values of $\kappa$. We consider $n=5000$ and the sampling interval $\Delta=1$. The results are represented on Figure \ref{Fig:Poiss}, we observe that the rates are small and stable regardless the value of $\kappa$ and the density considered.

\begin{figure}\begin{center}
\includegraphics[width=10cm,height=5cm]{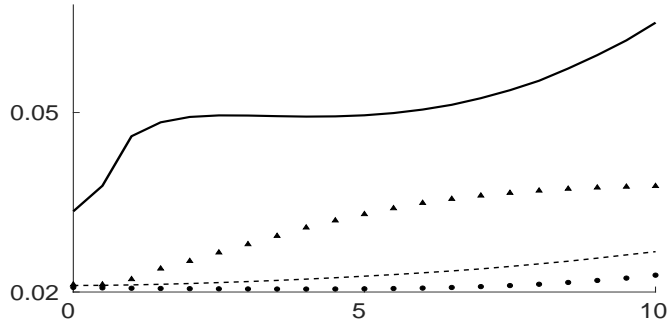}\end{center}
\caption{\label{Fig:Poiss}\footnotesize{\textbf{Decompounding:} Computations by $M=1000$ Monte Carlo iterations of the $\lk^2$-risks ($y$ axis) for different values of $\kappa\leq\frac12(\sqrt{n}-1)\sqrt{\log(n)}^{-1}=11.9$ ($x$ axis). Estimation of $f$ from $n=5000$ ($T=
5000$ and $\Delta=1$) increments of a compound Poisson process with intensity $\lambda=1$ and jump density $f$: Uniform $\mathcal{U}[1,3]$ (plain line), Gaussian $\mathcal{N}(2,1)$ (dots),  Gamma $\Gamma(2,1)$ (dotted line) and the mixture $0.7\mathcal{N}(4,1)+0.3 \Gamma(2,\frac{1}{2})$ (triangles). }}
\end{figure}

\section{Concluding remarks}

\paragraph*{Comments on the adaptive procedure.} In the present paper we develop an adaptive procedure that was successfully used in of Duval and Kappus \cite{duval2017nonparametric} which considers the problem of grouped data estimation. One observes i.i.d. realizations of $Y_{j}=X^{(1)}_{j}+\ldots+X^{(K)}_{j}$ where $K$ is a fixed and known integer and where the random variables $(X_{j}^{(1)},\ldots,X_{j}^{(K)})_{1\leq j\leq n}$ are i.i.d. This problem is a particular deconvolution problem where the density of the noise is unknown and depends on the density of interest. Here we adapt the procedure to two other classical inverse problems: deconvolution and decompounding. In each cases the resulting adaptive estimator is proven rate optimal, up to a logarithmic factor, for the $\lk^{2}$-risk.

Both in the grouped data setting (see  \cite{duval2017nonparametric}) and in the deconvolution setting (see Section \ref{sec:exD}) the computation of the adaptive cutoff, after simplifications, involves the set \begin{align}\hat m_{n}\in\big\{|\hat\phi_{Y,n}(u)|=\frac1{\sqrt{n}}(1+\kappa\sqrt{\log n})\big\},\quad \kappa>0\label{mdiscuss}\end{align}
 and not the characteristic function of $X_{1}$, nor the one of the errors in the deconvolution problem, nor the inverse of the operator relation $\phi_{Y}$ to $\phi_{X}$. However there are some differences between the grouped data setting where a uniform control on the characteristic function was needed and the deconvolution framework.  We only have pointwise control here, this difference is due to the fact that in the grouped data setting the density $f$ both played the role of the quantity of interest and the density of the noise; in \cite{duval2017nonparametric} we needed to bound it from above and below.
 
 In the decoumpounding setting the computation of the adaptive cutoff involves the em\-pi\-rical characteristic function of the jump density, which is more challenging than in the latter case. It does not seem obvious to have a simplified version as \eqref{mdiscuss} even though numerically selecting $\hat m_{n}$ this way seems relevant.

\paragraph*{Decompounding setting.} In the present paper we exhibit an adaptive rate optimal (up to a logarithmic factor) estimator of the jump density of a compound Poisson process from the discrete observation of its increments at sampling rate $\Delta=\Delta_{n}\to[0,\infty).$ It allows $\Delta_{n}\to \infty$ such that $\Delta_{n}<1/4\log(n\Delta_{n})$, our estimator remains consistent and optimal up to a logarithmic factor in some cases, e.g. if $\Delta_{n}=O\big(\log \log(n\Delta_{n})\big)$. Using \cite{duval2014no}, consistent non-parametric estimation of the jump density is impossible if $\exists \delta>0,\ \Delta_{n}=O(n\Delta_{n})^{\delta}$, the remaining questions are what happens in between and if the log loss in the upper bound that appears when $\Delta_{n}\to\infty$ is avoidable or not. The constant $1/4$ in the constant $\Delta_{n}<1/4\log(n\Delta_{n})$ of Theorem \ref{thm Poisson} can probably be improved.

	\section{Proofs\label{sec:proof}}

\subsection{Proof of Theorem \ref{thm:AD}}
Let $m\geq 0$ be fixed. First, consider the event  $\mathcal{E}=\{\hat{m}_{n} <m\}$, on this event we control the surplus in the bias of the estimator $\tilde f_{\hat m_{n}}$. Using the inequality 
\[
|\phi_X|^2   \leq 2 \frac{|\hat{\phi}_{Y,n}|^2}{|\phi_\eps|^2} 
+ 2\frac{ |\phi_Y- \hat{\phi}_{Y,n}|^2}{|\phi_\eps|^2} ,
\]
along with the definition of $\hat{m}_{n}$, gives 
\begin{align}
 \E\Big[\mathds{1}_{\mathcal{E}}\hspace{-0.5cm}  \int_{|u|\in[\hat{m}_{n},m]} \limits \hspace{-0.5cm} |\phi_X(u)|^2 \d u  \Big]
&\leq 2 \E\Big[ \mathds{1}_{\mathcal{E}}\hspace{-0.5cm}  \int_{|u|\in[\hat{m}_{n},m]} \limits\hspace{-0.5cm} \frac{  \kappa_n^2 n^{-1} }{|\phi_\eps(u)|^2} \d u\,      \Big] +   \int_{-m}^{m}
\frac{  \E[ |\hat{\phi}_{Y,n}(u)- \phi_Y(u)|^2 ] } {|\phi_\eps(u)|^2 } \d u\nonumber  \\
&\leq    \int_{-m}^{m} \frac{  2(\kappa_n^2+1) n^{-1} }{|\phi_\eps(u)|^2} \d u.\nonumber
\end{align}
Recall that $\kappa_{n}=1+\kappa\sqrt{\log(n)}$ and \eqref{eq:UBD}. This implies immediately that, on the event $\mathcal{ E}$, for a positive constant $C$ depending only on the choice of $\kappa$, 
\[
\E[  \|\tilde{f}_{\hat{m}_{n} }- f\|^2 \mathds{1}_\mathcal{E}  ] \leq   \|{f}_{{m} }- f\|^2+ C \frac{\log n}n \int_{-m}^{m}  \frac{ \d u.}{|\phi_\eps(u)|^2}.
\]
Second,  consider the complement set $\mathcal{E}^c$, where we control the surplus in the variance of $\tilde f_{\hat m_{n}}$.  By the definition of $\hat{m}_{n}$, it holds
\begin{align*}
 \E  \Big[ \int_{|u|\in[m,\hat{m}_{n}] }\limits \hspace{-0.5cm}  |\tilde{\phi}_{X,n}(u)&- \phi_X(u)|^2 \mathds{1}_{\{ |\phi_Y(u)| > n^{-1/2} \} } \d u  \mathds{1}_{\mathcal{E}^c}  \Big]
  \\&\leq   \int_{|u|\in[m,n^\alpha] }\limits \hspace{-0.5cm}  \frac{ \E[|\tilde{\phi}_{Y,n}(u)- \phi_Y(u)|^2]}{|\phi_\eps(u)|^2} \mathds{1}_{\{ |\phi_Y(u)| > n^{-1/2} \}} \d u.
\end{align*} On the event $\{ |\phi_Y(u)| > n^{-1/2} \}$, we derive that 
\[ \E[|\tilde{\phi}_{Y,n}(u)- \phi_Y(u)|^2]\leq|\phi_{Y}(u)|^{2}+ \E[|\hat{\phi}_{Y,n}(u)- \phi_Y(u)|^2] \leq|\phi_{Y}(u)|^{2}+\frac1n\leq 2|\phi_{Y}(u)|^{2}. \]
Consequently, we get
\[
\E  \Big[ \hspace{-0.5cm} \int_{|u|\in[m,\hat{m}_{n}] } \limits\hspace{-0.5cm}  |\tilde{\phi}_{X,n}(u)- \phi_X(u)|^2\,  \mathds{1}_{ \{|\phi_Y(u)| > n^{-1/2}\}}  \d u \,  \mathds{1}_{\mathcal{E}^c}  \Big]   \leq 2  \int_{[-m,m]^{c}} |\phi_X(u)|^2 \d u. 
\]Next, using that $|\tilde{\phi}_{X,n}(u)|\leq 1$, we derive that \begin{align*}
\E  \Big[ \hspace{-0.5cm} \int_{|u|\in[m,\hat{m}_{n}] } \limits\hspace{-0.5cm}  |&\tilde{\phi}_{X,n}(u)- \phi_X(u)|^2 \mathds{1}_{\{|\phi_Y(u)| \leq  n^{-1/2}\}}\d u \,  \mathds{1}_{\mathcal{E}^c}  \Big] \\
\leq &   \int_{|u|\in[m,n^\alpha] }\limits  \hspace{-0.5cm}  |\phi_X(u)|^2 \d u  + 4 \hspace{-0.5cm} \int_{|u|\in[m,n^\alpha] }\limits \hspace{-0.5cm} \PP(|\hat\phi_{Y,n}(u)|\geq \kappa_{n}n^{-1/2})\mathds{1}_{\{|\phi_{Y}(u)|\leq n^{-1/2}\}} \d u  \\
\leq &   \int_{|u|\in[m,n^\alpha] }\limits  \hspace{-0.5cm}  |\phi_X(u)|^2 \d u  + 4\hspace{-0.5cm}  \int_{|u|\in[m,n^\alpha] }\limits\hspace{-0.5cm}  \PP( |\hat{\phi}_{Y,n}(u) - \phi_Y(u)| > \kappa (\log n/n)^{1/2}  )   \d u\\   \leq& \int_{u\in[m,m]^{c} }\limits\hspace{-0.5cm}   |\phi_X(u)|^2 \d u +  8  n^{ \alpha - \kappa^2/2}. 
 \end{align*}The last inequality is a direct consequence of the Hoeffding inequality. 
Putting the above together, we have shown that for universal positive constants $C_1$ and  $C_3$ and  a constant $C_2$ depending only on $ \kappa$, for all $m\geq 0$,
\[
\E[  \| \tilde{f}_{\hat m_{n}}- f\|^2  ]  \leq  C_1 \|f-f_{m}\|^{2}   + C_2\frac{ \log n}n \int_{-m }^{m}  \frac{{\d u} }{|\phi_\eps(u)|^2}   +  C_3 n^{\alpha - \kappa^2/2} .   
\]
Taking the infimum over $m$ completes the proof.   \hfill$ \Box$

\subsection{Proof of Theorem \ref{thm Poisson}}
Proof of Theorem \ref{thm Poisson} uses similar arguments as the proof of Theorem 1 of Duval and Kappus \cite{duval2017nonparametric}. However, estimator \eqref{eq:fhatP} is different from the estimator studied in  \cite{duval2017nonparametric} and we need to take into account the additional parameter $\Delta$ that needs to be carefully handled.

\subsubsection{Preliminaries}

We establish two technical Lemmas used in the proof of Theorem \ref{thm Poisson}.

\begin{lemma}\label{lem Omega}
Let $m>0$ and $\zeta>0$ and define the event
\[
\Omega_{\zeta,\Delta}(m):=\Big\{\forall u\in[-m,m],\ \big|\hat\phi_{\Delta,n}(u)-\phi_{\Delta}(u)\big|\leq \zeta\sqrt{\frac{\log(n\Delta)}{n\Delta}}\Big\}.
\]
\begin{enumerate}
\item
If $\E[X_{1}^{2}]$ is finite, then, the following holds for $\eta>0$ and any $\zeta>\sqrt{\Delta(1+2\eta)}$,
\[\PP\big(\Omega_{\zeta,\Delta}(m)^{c}\big)\leq \frac{\E[X_{1}^{2}]}{n\Delta}+4\frac {m}{(n\Delta)^{\eta}}.\]
\item
If $\E[X_{1}^{4}]$ is finite, then, the following holds for $\eta>0$ and any $\zeta>\sqrt{\Delta(1+2\eta)}$,
\[\PP\big(\Omega_{\zeta,\Delta}(m)^{c}\big)\leq \frac{C}{(n\Delta)^{2}}+4\frac {m}{(n\Delta)^{\eta}},\]
\end{enumerate} where $C$ depends on $\E[X_{1}^{4}]$ and $\E[X_{1}^{2}]$.
\end{lemma}

\begin{proof}[Proof of Lemma \ref{lem Omega}]
Consider the events
\begin{align*}
A(c):=&\Big\{\Big|\frac{1}{n}\sum_{j=1}^{n}|Y_{j\Delta}|-\E[|Y_{\Delta}|]\Big|\leq c\Big\}\\
B_{h,\tau}(m):=&\Big\{\forall |k|\leq \Big\lceil\frac{m}{h}\Big\rceil,\ \big|\hat\phi_{\Delta,n}(kh)-\phi_{\Delta}(kh)\big|\leq \tau\sqrt{\frac{\log(n\Delta)}{n\Delta}}\Big\}
\end{align*} for some positive constants $c,h$ and $\tau$ to be determined. First, using that $x\rightarrow e^{iux}$ is 1-Lipschitz and that $\E[|Y_{\Delta}|]\leq \Delta\E[|X_{1}|]$ we get on the event $A(c)$ 
\begin{align}
\label{eq lemO1}\big|\hat\phi_{\Delta,n}(u)-\hat\phi_{\Delta,n}(u+h)\big|\mathds{1}_{A(c)}\leq h\big(\Delta\E[|X_{1}|]+c\big),\quad \forall u\in\R,\ h>0.
\end{align} If $\E[X_{1}^{2}]$ is finite the Markov inequality and the bound $\V[|Y_{\Delta}|]\leq \V[Y_{\Delta}]= \Delta\E[X_{1}^{2}]$ lead to 
\begin{align}
\label{eq lemO2}\PP\big(A(c)^{c}\big)\leq \frac{\Delta\E[X_{1}^{2}]}{c^{2}n}.
\end{align} If $\E[X_{1}^{4}]$ is finite \eqref{eq lemO2} can be improved using that \[\E\Big[\Big(\sum_{j=1}^{n}(|Y_{j\Delta}|-\E[|Y_{\Delta}|])\Big)^{4}\Big]\leq n\Delta^{2}\E[X_{1}^{4}]+3n(n-1)\Delta^{2}\E[X_{1}^{2}] \], leading to 
\begin{align}
\label{eq lemO3}\PP\big(A(c)^{c}\big)\leq C\frac{\Delta^{2}}{c^{4}n^{2}},
\end{align} where $C$ is a constant depending on $\E[X_{1}^{4}]$ and $\E[X_{1}^{2}]$.
 Second, we have that
\begin{align}
\PP\big(B_{h,\tau}(m)^{c}\big)&\leq\PP\bigg(\exists\ |k|\leq \Big\lceil\frac{m}{h}\Big\rceil,\ \big|\hat\phi_{\Delta,n}(kh)-\phi_{\Delta}(kh)\big|> \tau\sqrt{\frac{\log(n\Delta)}{n\Delta}}\bigg)\nonumber\\
&\leq \sum_{k=-\lceil{m}/{h}\rceil}^{\lceil{m}/{h}\rceil}\PP\Big(\big|\hat\phi_{\Delta,n}(kh)-\phi_{\Delta}(kh)\big|> \tau\sqrt{\frac{\log(n\Delta)}{n\Delta}}\Big)\nonumber\\
&\leq \sum_{k=-\lceil{m}/{h}\rceil}^{\lceil{m}/{h}\rceil} 2\exp\Big(-\frac{\tau^{2}\log(n\Delta)}{2\Delta}\Big)=4\Big\lceil\frac{m}{h}\Big\rceil (n\Delta)^{-\tau^{2}/(2\Delta)}\nonumber
\end{align}where the last inequality is obtained applying the Hoeffding inequality. 
Let $|u|\leq m$, there exists $k$ such that $u\in[kh-\frac{h}{2},kh+\frac{h}{2}]$ and we can write that
\begin{align*}
\mathds{1}_{A(c)\cap B_{h,\tau}(m)}\big|\hat\phi_{\Delta,n}(u)-\phi_{\Delta}(u)\big|&\leq \mathds{1}_{A(c)\cap B_{h,\tau}(m)}\Big(\big|\hat\phi_{\Delta,n}(u)-\hat\phi_{\Delta,n}(kh)\big|+\big|\hat\phi_{\Delta,n}(kh)-\phi_{\Delta}(kh)\big|\\&\hspace{3cm}+\big|\phi_{\Delta}(kh)-\phi_{\Delta}(u)\big|\Big).
\end{align*}
Using \eqref{eq lemO1}, the definition of $B_{h,\tau}(m)$ and that $x\rightarrow e^{iux}$ is 1-Lipschitz, lead to
\begin{align}\label{eq lemO4}
\mathds{1}_{A(c)\cap B_{h,\tau}(m)}\underset{u\in[-m,m]}{\sup}\big|\hat\phi_{\Delta,n}(u)-\phi_{\Delta}(u)\big|&\leq 2h\Delta\E[|X_{1}|]+hc+\tau\sqrt{\frac{\log(n\Delta)}{n\Delta}}.
\end{align}
Taking $c=\Delta$, $h=o\big(\sqrt{\frac{\log(n\Delta)}{n\Delta}}\big)$ such that $h>1/\sqrt{n\Delta}$ and $\zeta>\tau$, \eqref{eq lemO4} shows that  $A(c)\cap B_{h,\tau}(m)\subset \Omega_{\zeta,\Delta}(m).$ Moreover, it follows from $h>1/\sqrt{n\Delta}$, \eqref{eq lemO1} and \eqref{eq lemO2} that, for all $\eta>0$
\begin{align*}
\PP\big(\Omega^{c}_{\zeta,\Delta}(m)\big)&\leq \PP\big(A^{c}(\Delta)\big)+\PP\big(B^{c}_{h,\tau}(m)\big)\leq \frac{\E[X_{1}^{2}]}{n\Delta}+4\Big\lceil\frac{m}{h}\Big\rceil (n\Delta)^{-\frac{\tau^{2}}{2\Delta}}\leq \frac{\E[X_{1}^{2}]}{n\Delta}+4m(n\Delta)^{\frac{\Delta-\tau^{2}}{2\Delta}}.
\end{align*} Finally, choosing $\tau^{2}=\Delta(1+2\eta)$ leads to the result. The second inequality is obtained follows from similar arguments using \eqref{eq lemO2} instead of \eqref{eq lemO3}.
\end{proof}

\begin{lemma}\label{lem DL}Let  $\gamma>0$, define 
\[M_{n,\Delta}^{ (\gamma)}:= \min\big\{m\geq 0: |\phi_{\Delta}(m)|= \gamma\sqrt{{\log(n\Delta)}/{(n\Delta)}} \big\},
 \]with the convention $\inf\{\emptyset\}=+\infty$.
Take $\gamma> \zeta>0$, then, we have 
\[
\mathds{1}_{|u|\leq M_{n,\Delta}^{\gamma}\wedge m ,\Omega_{\zeta,\Delta}(m)}\Big|\Log(\hat\phi_{\Delta,n}(u))-\Log(\phi_{\Delta}(u))\Big| \leq \frac{\gamma}{\zeta} \log\Big(\frac{\gamma}{\gamma-\zeta}\Big)\frac{|\hat\phi_{\Delta,n}(u)-\phi_{\Delta}(u)|}{|\phi_{\Delta}(u)|}.\]
\end{lemma}

\begin{proof}[Proof of Lemma \ref{lem DL}]
First note that for $|u|\leq M_{n,\Delta}^{\gamma} ,$ the ratio $\frac{\phi'_{\Delta}}{\phi_{\Delta}}$ is well defined. Moreover, on the event $\Omega_{\zeta,\Delta}(m)$ then we have that $$|\hat\phi_{\Delta,n}(u)|\geq|\phi_{\Delta}(u)|-|\hat\phi_{\Delta,n}(u)-\phi_{\Delta}(u)|\geq (\gamma-\zeta)\sqrt{\frac{\log(n\Delta)}{n\Delta}}>0,\quad \forall |u|\leq m.$$ Then, the quantity $\frac{\hat\phi'_{\Delta,n}}{\hat\phi_{\Delta,n}}$ is also well defined if $\gamma>\zeta$. For $v\in \R$, notice that \begin{align}
\frac{\dephi_{\Delta,n}(v)}{\ephi_{\Delta,n}(v)}  -\frac{\dphi_{\Delta}(v)}{\phi_{\Delta}(v)} \label{eq:der}=& \dfrac{   \Big(- \frac{(\ephi_{\Delta,n}(v)-\phi_{\Delta}(v))}{\phi_{\Delta}(v)}   \Big)' }{\Big(1- \frac{(\ephi_{\Delta,n}(v)-\phi_{\Delta}(v))}{\phi_{\Delta}(v)}    \Big)   }.
\end{align} On the event $\Omega_{\zeta,\Delta}(m)$, it holds $\forall u\in [-m\wedge  M_{n,\Delta}^{\gamma} , m \wedge M_{n,\Delta}^{\gamma} ] $
\begin{equation}  \label{Bm}
|\hat{\phi}_{\Delta,n} (u) -  \phi_{\Delta}(u) | \leq \zeta \sqrt{\log (n\Delta)} (n\Delta)^{-\frac{1}{2} } \leq \frac{\zeta}{\gamma}  |\phi_{\Delta}(u) |, 
\end{equation} where $\gamma>\zeta$.
Then, a  Neumann series expansion, with \eqref{eq:der} and $\eqref{Bm}$ gives for $|v|\leq m\wedge  M_{n,\Delta}^{\gamma} $, 
\begin{align*}
  \frac{\dephi_{\Delta,n}(v)}{\ephi_{\Delta,n}(v)}  -\frac{\dphi_{\Delta}(v)}{\phi_{\Delta}(v)}  =- \sum_{\ell=0}^{\infty} \Big( \frac{\ephi_{\Delta,n}(v)-\phi_{\Delta}(v)}{\phi_{\Delta}(v)}   \Big)' \Big( \frac{\ephi_{\Delta,n}(v)-\phi_{\Delta}(v)}{\phi_{\Delta}(v)}    \Big)^{\ell},
\end{align*}  
where
\begin{align*}   \Big( \frac{\ephi(v)-\phi_{\Delta}(v)}{\phi_{\Delta}(v)}   \Big)' \Big( \frac{\ephi_{\Delta,n}(v)-\phi_{\Delta}(v)}{\phi_{\Delta}(v)}    \Big)^{\ell} = \frac{1}{\ell +1} \Big[ \Big( \frac{\ephi_{\Delta,n}(v)-\phi_{\Delta}(v)}{\phi_{\Delta}(v)}    \Big)^{\ell+1}\Big]'.
\end{align*}
Using $\hat{\phi}_{\Delta}(0)- \phi_{\Delta}(0)=0$ and \eqref{Bm}, we get 
\begin{align}
 \mathds{1}_{|u|\leq m\wedge M_{n,\Delta}^{\gamma} ,\Omega_{\zeta,\Delta}(m)}\Big| \int_{0}^{u} \Big( \frac{\dephi_{\Delta,n}(v)}{\ephi_{\Delta,n}(v)}  -\frac{\dphi_{\Delta}(v)}{\phi_{\Delta}(v)}   \Big) \d v  \Big|&
\leq \sum_{\ell=0}^{\infty} \frac{1}{\ell+1}   \frac{|\hat{\phi}_{\Delta,n}(u) - \phi_{\Delta}(u) |^{\ell +1}  }{|\phi_{\Delta}(u)|^{\ell+1}  }\nonumber\\
 &\leq   \frac{|\hat{\phi}_{\Delta,n}(u) - \phi_{\Delta}(u) |}{|\phi_{\Delta}(u)|  }  \sum_{\ell=0}^{\infty} \frac{\big(\zeta/\gamma\big)^{\ell} }{\ell+1}  \nonumber\\&= \frac{\gamma}{\zeta} \log\big(\frac{\gamma}{\gamma-\zeta}\big)\frac{|\hat{\phi}_{\Delta,n}(u) - \phi_{\Delta}(u) |}{|\phi_{\Delta}(u)|  } , \label{eq lemDL}
\end{align}  
which completes the proof.
\end{proof}

\subsubsection{Proof of Theorem \ref{thm Poisson}}
We have the decomposition $$\| \widehat{f}_{m,\Delta}-f\|^2=\|f_m-f\|^2+\| \widehat{f}_{m,\Delta}-f_m\|^2 = \|f_m-f\|^2+\frac{1}{2\pi}\int_{- m}^{ m} |\widetilde\phi_{n}(u)-\phi(u)|^2\d u.$$ Let $\gamma>\zeta$, we decompose the second term on the events $\{m\leq M_{n,\Delta}^{\gamma} \}$ and $\Omega_{\zeta,\Delta}(m)$ of  Lemma \ref{lem DL},\begin{align*}
\intm |\widetilde\phi_{n}(u)-\phi(u)|^2\d u&=\int_{-m\wedge M_{n,\Delta}^{\gamma} }^{m\wedge M_{n,\Delta}^{\gamma} }\mathds{1}_{\Omega_{\zeta,\Delta}(m)}|\widetilde\phi_{n}(u)-\phi(u)|^2\d u\\&\quad+\mathds{1}_{m> M_{n,\Delta}^{\gamma},\Omega_{\zeta,\Delta}(m)}\int_{|u|\in [M_{n,\Delta}^{\gamma} ,m]}|\widetilde\phi_{n}(u)-\phi(u)|^2\d u\\&\quad+\mathds{1}_{\Omega_{\zeta,\Delta}(m)^{c}}\intm|\widetilde\phi_{n}(u)-\phi(u)|^2\d u\\
:&= T_{1,n}+T_{2,n}+T_{3,n}.
\end{align*}
Fix $\gamma_{\Delta}=\frac{2\zeta}{1\wedge\Delta}>\zeta$. On the event $\{|u|\leq m\wedge M_{n,\Delta}^{\gamma_{\Delta}},\Omega_{\zeta,\Delta}(m)\}$, Lemma \ref{lem DL} and equations \eqref{eq:Ff}, \eqref{eq fhatstar1}  and \eqref{eq lemDL}, along with \eqref{Bm}, imply \[\big|\hat\phi_{n}(u)\big|\leq 1+\frac{|\Log(\hat\phi_{\Delta,n}(u))-\Log(\phi_{\Delta}(u))|+|\Log(\phi_{\Delta}(u))|}{\Delta}\leq 3+\frac{1}{\Delta}\log\Big(\frac{\gamma}{\gamma-\zeta}\Big)\leq 4,\]
 consequently $\widetilde\phi_{n}(u)=\hat\phi_{n}(u)$. Then, we get from Lemma \ref{lem DL} and the definition of $\gamma_{\Delta}$, that
\begin{align*}\E[T_{1,n}]&=\frac{1}{\Delta^{2}}\int_{-m\wedge  M_{n,\Delta}^{\gamma_{\Delta}} }^{m\wedge  M_{n,\Delta}^{\gamma_{\Delta}} }\E\Big[\mathds{1}_{\Omega_{\zeta,\Delta}(m)}\big|\Log(\hat\phi_{\Delta,n}(u))-\Log(\phi_{\Delta}(u))\big|^{2}\Big]\d u\\ & 
\leq \frac{1}{\Delta^{2}}\int_{-m}^{m}\frac{\E\big[|\hat\phi_{\Delta,n}(u)-\phi_{\Delta}(u)|^{2}\big]}{|\phi_{\Delta}(u)|^{2}}\d u.
\end{align*} Direct computations together with the Lévy-Kintchine formula lead to
\[\E\big[|\hat\phi_{\Delta,n}(u)-\phi_{\Delta}(u)|^{2}\big]=\frac{1-|\phi_{\Delta}(u)|^{2}}{n}\leq \frac{\big(2\Delta|\mbox{Re}(\phi(u))-1|\big)\wedge 1}{n}\leq \frac{2\Delta\wedge 1}{n}.\]
We derive that \[
\E[T_{1,n}]\leq \frac{2\Delta\wedge 1}{n\Delta^{2}}\intm\frac{\d u}{|\phi_{\Delta}(u)|^{2}}\leq \frac{2}{n\Delta}\intm\frac{\d u}{|\phi_{\Delta}(u)|^{2}}.\]
Next, fix $\zeta>\sqrt{5\Delta}$, Lemma \ref{lem Omega} with $\eta=2$ gives
\begin{align*}
\E[T_{3,n}]&\leq 2 \ 5^{2}m\Big( \frac{\E[X_{1}^{2}]}{n\Delta}+4\frac{m}{(n\Delta)^{2}}\Big).\end{align*} Moreover, using  $|\phi_{\Delta}(u)|\geq e^{-2\Delta}$, $\forall u\in\R$ together with the constraint $\Delta\leq \delta\log(n\Delta)$, $\delta<\frac{1}{4}$, we get \[|\phi_{\Delta}(u)|\geq (n\Delta)^{-2\delta}>\gamma_{\Delta}\sqrt{\log(n\Delta)/(n\Delta)}, \quad\forall u\in \R.\] Finally, $M_{n,\Delta}^{\gamma_{\Delta}}=+\infty$, $\forall \zeta>0$ and $T_{2,n}=0$ almost surely. Gathering all terms completes the proof. \hfill$\Box$

\subsection{Proof of Theorem \ref{thm:AP}}
Let $m\geq 0$ be fixed. Consider the event  $\mathcal{E}=\{\hat{m}_{n} <m\}$, on this event we control the surplus in the bias of the estimator $\tilde f_{\hat m_{n}}$. Using the inequality 
$
|\phi|^2   \leq 2 {|\tilde{\phi}_{n}|^2}
+ 2{ |\phi- \tilde{\phi}_{n}|^2},
$
along with the definition of $\hat{m}_{n}$ and Theorem \ref{thm Poisson}, gives 
\begin{align}
 \E\Big[\mathds{1}_{\mathcal{E}} \hspace{-0.5cm}  \int_{|u|\in[\hat{m}_{n},m]} \limits\hspace{-0.5cm}  |\phi(u)|^2 \d u  \Big]
&\leq 2 \E\Big[\mathds{1}_{\mathcal{E}}  \hspace{-0.5cm}  \int_{|u|\in[\hat{m}_{n},m]} \limits\hspace{-0.5cm}  \frac{  \kappa_{n,\Delta}^2}{n\Delta} \d u  \Big] + 2  \int_{-m}^{m}
  \E[ |\tilde{\phi}_{n}(u)- \phi(u)|^2 ]  \d u\nonumber  \\
&\leq  4\frac{\kappa_{n,\Delta}^{2}m}{n\Delta} +\frac{4}{{n\Delta}}\int_{-m}^{m}\frac{\d u}{|\phi_{\Delta}(u)|^{2}}+2^{2} 5^{2}\E[X_{1}^{2}]\frac{m}{n\Delta}+2^{4}5^{2} \frac{{m}^{2}}{(n\Delta)^{2}}.\nonumber
\end{align}
Recall that $\kappa_{n,\Delta}={e^{2\Delta}+\kappa\sqrt{\log(n\Delta)}}$ together with Theorem \ref{thm Poisson}, this implies immediately that, on the event $\mathcal{ E}$, for a positive constant $C$ depending on the choice of $\kappa$ and $\E[X_{1}^{2}]$, 
\[
\E[  \|\overline{f}_{\hat{m}_{n},\Delta }- f\|^2 \mathds{1}_\mathcal{E}  ] \leq   \|{f}_{{m} }- f\|^2+ C \frac{{\log(n\Delta)}m}{n\Delta} +\frac{5}{{n\Delta}}\int_{-m}^{m}\frac{\d u}{|\phi_{\Delta}(u)|^{2}}+2^{5}5^{2} \frac{{m}^{2}}{(n\Delta)^{2}} .
\]Second,  consider the complement set $\mathcal{E}^c$, where we control the surplus in the variance of $\tilde f_{\hat m_{n}}$.  By the definition of $\hat{m}_{n}$, it holds
\[
 \E  \Big[  \hspace{-0.5cm} \int_{|u|\in[m,\hat{m}_{n}] }\limits  \hspace{-0.5cm}  |\overline{\phi}_{n}(u)- \phi(u)|^2  \d u  \mathds{1}_{\mathcal{E}^c}  \Big]
  \leq   \int_{|u|\in[m,(n\Delta)^\alpha] }\limits   \hspace{-0.5cm} { \E[|\overline{\phi}_{n}(u)- \phi(u)|^2]}\d u.
\]  Let $\eta>2$, such that $\alpha-\eta<-1$, and $\zeta>\sqrt{\Delta(1+2\eta)}$ and $\gamma_{\Delta}$ as in the proof of Theorem \ref{thm Poisson} (leading to $M_{n,\Delta}^{\gamma_{\Delta}}=+\infty$). Then, Lemmas \ref{lem Omega} (decomposing on $\Omega_{\zeta,\Delta}((n\Delta)^\alpha)$) and \ref{lem DL} lead to \[ \E[|\overline{\phi}_{n}(u)- \phi(u)|^2]\leq|\phi(u)|^{2}+ \E[|\hat{\phi}_{n}(u)- \phi(u)|^2] \leq|\phi(u)|^{2}+\frac{2}{n\Delta|\phi_{\Delta}(u)|^{2}}+\frac{\E[X_{1}^{2}]}{n\Delta}+4\frac{1}{n\Delta}. \]
First, on the event $\{ |\phi(u)| > e^{2\Delta}/\sqrt{n\Delta} \}$, we obtain 
\[ \E[|\overline{\phi}_{n}(u)- \phi(u)|^2]\leq|\phi(u)|^{2}\big(6+\E[X_{1}^{2}]\big). \]
Consequently, define $C_{0}:=6+\E[X_{1}^{2}]$, then,
\begin{align*}
 \E  \Big[\hspace{-0.5cm} \int_{|u|\in[m,\hat{m}_{n}] }\limits  \hspace{-0.5cm}  |\overline{\phi}_{n}(u)- \phi(u)|^2  \,  \mathds{1}_{ \{|\phi(u)| > e^{2\Delta}/\sqrt{n\Delta}\}}  \d u \,  \mathds{1}_{\mathcal{E}^c}  \Big]   \leq C_{0} \hspace{-0.5cm}\int_{[-m,m]^{c}} \limits|\phi(u)|^2 \d u. 
\end{align*}
Next, using that $|\overline{\phi}_{n}(u)|\leq 4$ and the definition of $\hat m_{n}$, we derive that \begin{align*}
 \E  \Big[\hspace{-0.5cm} \int_{|u|\in[m,\hat{m}_{n}] }\limits  \hspace{-0.5cm}  |\overline{\phi}_{n}(u)&- \phi(u)|^2  \,  \mathds{1}_{ \{|\phi(u)| \leq e^{2\Delta}/\sqrt{n\Delta}\}}  \d u \,  \mathds{1}_{\mathcal{E}^c}  \Big]    \\
\leq &   \hspace{-0.5cm} \int_{|u|\in[m,(n\Delta)^\alpha] }\limits   \hspace{-0.5cm} |\phi(u)|^2\d u  + 5^{2} \hspace{-0.5cm} \int_{|u|\in[m,(n\Delta)^\alpha] }\limits  \hspace{-0.5cm}\PP\big(|\hat \phi_{n}(u)|\geq \kappa_{n,\Delta}/\sqrt{n\Delta}\big)\mathds{1}_{\{|\phi(u)|\leq e^{2\Delta}/\sqrt{n\Delta}\}}  \d u  \\
\leq &  \hspace{-0.5cm}  \int_{|u|\in[m,(n\Delta)^\alpha] }\limits  \hspace{-0.5cm}  |\phi(u)|^2 \d u  + 5^{2}\hspace{-0.5cm} \int_{|u|\in[m,(n\Delta)^\alpha] }\limits \hspace{-0.5cm} \PP\big( |\hat\phi_{n}(u)- \phi(u)| > \kappa \sqrt{\log (n\Delta)/(n\Delta)}  )   \d u\\   \leq& \int_{u\in[m,m]^{c} }\limits  |\phi(u)|^2 \d u +  5^{2}T_{n}. 
 \end{align*} Finally, we give a bound for $T_{n}$ using Lemmas \ref{lem Omega} and \ref{lem DL} with $\gamma_{\Delta}$ and $\zeta>0$ as above,
 \begin{align*}
\PP\Big( |\hat{\phi}_{n}(u)-\phi(u)| \geq {\kappa}&\sqrt{\frac{\log (n\Delta)}{n\Delta}}\Big)
 =   \PP\Big( |\Log(\hat{\phi}_{\Delta,n}(u))-\Log(\phi_{\Delta}(u))| \geq {\kappa\Delta}\sqrt{\frac{\log (n\Delta)}{n\Delta}}\Big)\\
&\leq  \PP\Big( |\hat{\phi}_{\Delta,n}(u) - \phi_\Delta(u)|\geq |\phi_{\Delta}(u)|{\kappa\Delta}\sqrt{\frac{\log (n\Delta)}{n\Delta}}\Big)+\PP\big(\Omega_{\zeta,\Delta}^{c}((n\Delta)^\alpha)\big).
\end{align*}Define $c({\Delta}):={\kappa{\Delta} e^{-2\Delta}}$, then, we derive from the Hoeffding inequality and Lemma \ref{lem Omega} that
\[
\PP\Big( |\hat{\phi}_{n}(u)-\phi(u)| \geq {\kappa}\sqrt{\frac{\log (n\Delta)}{n\Delta}}\Big)\leq {2}{(n\Delta)}^{-c({\Delta})^{2}}+\frac{C}{(n\Delta)^{2}}+4(n\Delta)^{\alpha-\eta} ,
\] where $C$ depends on $\E[X_{1}^{2}]$ and $\E[X_{1}^{4}]$.
Fix $\eta>3$, such that $2\alpha-\eta<-1$ and $\zeta>\sqrt{\Delta(1+2\eta)}$, it follows that
\[
T_{n}\leq {4}{(n\Delta)}^{\alpha-c({\Delta})^{2}}+\frac{C^{\prime}}{n\Delta},
\]where $C^{\prime}$ depends on $\E[X_{1}^{2}]$ and $\E[X_{1}^{4}]$.
Putting the above together, we have shown that for a positive constant $C_{1}$, depending on $\kappa$, $\E[X_{1}^{2}]$ and $\E[X_{1}^{4}]$, and $C_2$ a constant depending on  $\E[X_{1}^{2}]$ and $\E[X_{1}^{4}]$
\begin{align*}
\E[  \| \overline{f}_{\hat m_{n},\Delta}- f\|^2  ]  \leq   C_{1}&\Big( \|{f}_{{m} }- f\|^2+  \frac{\log(n\Delta)m}{n\Delta} +\frac{1}{{n\Delta}}\int_{-m}^{m}\frac{\d u}{|\phi_{\Delta}(u)|^{2}}+ \frac{{m}^{2}}{(n\Delta)^{2}}\Big)\\&+C_{2}\Big(\frac{1}{n\Delta}+(n\Delta)^{\alpha-c(\Delta)^{2}}\Big).   
\end{align*} Taking the infimum in $m$ completes the proof.   \hfill$ \Box$

\printbibliography

\end{document}